\documentclass[reqno]{amsart}

\usepackage[T1]{fontenc}
\usepackage{amssymb}
\usepackage{hyperref}
\usepackage{amsrefs}
\usepackage[table]{xcolor}
\usepackage{enumerate}
\usepackage{tikz-cd}
\usepackage{commath}

\newtheorem{theorem}{Theorem}
\newtheorem{lemma}[theorem]{Lemma}
\newtheorem{proposition}[theorem]{Proposition}
\newtheorem{corollary}[theorem]{Corollary}

\theoremstyle{definition}
\newtheorem{definition}[theorem]{Definition}
\newtheorem{example}[theorem]{Example}

\theoremstyle{remark}

\numberwithin{theorem}{section}
\numberwithin{equation}{section}

\newcommand{\F}{\mathcal{F}}
\newcommand{\pr}{\mathbf{pr}}
\renewcommand{\Pr}{\mathbf{Pr}}
\renewcommand{\O}{\mathcal{O}}

\newcommand{\G}{\mathcal{G}}

\newcommand{\Id}{\mathtt{Id}}
\renewcommand{\H}{\mathcal{H}}
\newcommand{\N}{\mathcal{N}}

\newcommand{\K}{\mathcal{K}}
\newcommand{\s}{\mathbf{s}}
\newcommand{\from}{\leftarrow}
\renewcommand{\t}{\mathbf{t}}

\DeclareMathOperator{\Hom}{Hom}

\newcommand{\isom}{\cong}
\newcommand{\inv}{^{-1}}
\newcommand{\m}{\mathbf{m}}
\newcommand{\C}{\mathcal{C}}

\renewcommand{\u}{\mathbf{u}}

\renewcommand{\i}{\mathbf{i}}
\newcommand{\g}{\mathfrak{g}}
\renewcommand{\til}{\widetilde}

\newcommand{\X}{\mathcal{X}}
\newcommand{\Y}{\mathcal{Y}}

\renewcommand{\L}{\mathcal{L}}
\renewcommand{\P}{\mathcal{P}}

\newcommand{\T}{\mathbb{T}}

\newcommand{\B}{\mathbf{B}}
\newcommand{\Man}{\mathtt{Man}}
\newcommand{\Isom}{\cong}
\newcommand{\DMan}{\mathtt{DMan}}

\title{Poisson manifolds and their associated stacks}
\author{Joel Villatoro}
\thanks{The author was partially supported by an AGEP-GRS fellowship under NSF grant DMS 130847.}
\email{villato2@illinois.edu}

\begin{document}
\begin{abstract}
We associate to any integrable Poisson manifold a stack, i.e.~a category fibered in groupoids over a site.
The site in question has objects Dirac manifolds and morphisms pairs consisting of a smooth map and a closed 2-form.
We show that two Poisson manifolds are symplectically Morita equivalent if and only if their associated stacks are isomorphic.
\end{abstract}
\maketitle
\bibliographystyle{spmpsci}      
\section{Introduction}
In differential geometry, differentiable stacks provide models for singular spaces.
Intuitively, a differentiable stack generalizes the notion of manifold, where atlas are replaced by Lie groupoids.
Two Lie groupoids define the same stack if they are Morita equivalent.
The stack itself can be thought of as the orbit space of a Lie groupoid, but in reality it encondes all the transversal geometry of the leaves.

It has long been known that a Poisson manifold has an associated Lie algebroid.
When this algebroid is integrable, the corresponding source 1-connected Lie groupoid is a  \emph{symplectic} groupoid.
Conversely, the space of objects of any symplectic groupoid has a natural Poisson structure.
In~\cite{Morping}, Ping Xu introduced a notion of \emph{symplectic} Morita equivalence for symplectic groupoids.
This paper answers the following question:
\begin{itemize}
\item[(Q)] What is the notion of stack associated with a symplectic groupoid and symplectic Morita equivalences?
\end{itemize}
Notice that since symplectic Morita equivalence is more strict than ordinary Morita equivalence, the answer to this question \emph{is not} the (ordinary) stack associated with the underlying Lie groupoid.

In order to explain our answer to this question, let us recall that in the study of morphisms and isomorphisms of differentiable stacks one uses three, roughly equivalent, languages, summarized in the following table.
\begin{center}
\begin{tabular}{|c|c|}
\hline\noalign{\smallskip}
                    & objects: Lie groupoids \\
Principal Bundles   & morphisms: left principal bibundles \\
                    & equivalences: principal bibundles \\
\noalign{\smallskip}\hline\noalign{\smallskip}
                    & objects: Lie groupoids \\
Calculus of Fractions & morphisms: formal factions \(\frac{F: \G' \to \G}{G:\G' \to \H{} }\)  \\
                    & equivalences: a pair of weak equivalences \(F/G\) \\
\hline
                    & objects: categories fibered in groupoids \\
Fibered Categories                 & morphisms: Fiber preserving functors  \\
                    & equivalence: equivalence of categories \\
\noalign{\smallskip}\hline\end{tabular}
\end{center}
\vspace{10pt}

In Xu's work~\cite{Morping}, two symplectic groupoids are Morita equivalent if there exists a \emph{symplectic} principal bibundle between them.
Hence, this notion is a natural extension to the symplectic setting of the notion of equivalence in the first language above.
However, extensions to the symplectic setting of many other concepts listed in this table seem to be absent from the literature, and are relevant to answer our main question.
For example:%
\begin{enumerate}[(Q1)]
\item What is a left principal bibundle of symplectic groupoids?
\item What is a weak equivalence of symplectic groupoids?
\item What is the fibered category associated with a symplectic groupoid?
\end{enumerate}
In this paper we give answers to these questions.
In particular, the answer to the last question will provide our answer to (Q) above: we will show that one can associate to a symplectic groupoid a fibered category over a certain site that incorporates also the symplectic geometry.

Although we will be mostly concerned with symplectic groupoids, and hence integrable Poisson manifolds, we will also consider non-integrable Poisson manifolds.
Our answer to the questions above lead to a natural notion of infinitesimal symplectic Morita equivalence, which is valid for any Poisson manifold.
So ultimately we will be able to provide an answer to the more general question:
\begin{itemize}
\item[(Q')] What is the notion of stack associated with a Poisson manifold?
\end{itemize}

We will proceed as follows.
In Section~\ref{section:dman} we establish our notation and we introduce a fundamental notion in our work: $\DMan$, the \emph{site of Dirac manfolds}. The reason for working in the more general setting of Dirac geometry is that Dirac manifolds are much better behaved categorically than Poisson manifolds (e.g., they admit pull backs).

In Section~\ref{section:dlie} we introduce a new object called a \emph{D-Lie groupoid}.
Briefly, a D-Lie groupoid is just a groupoid internal to $\DMan$.
Intuitively, they can be thought of as `pseudo-integrations' of Dirac structures.
We will see that the notion of D-Lie groupoid captures many interesting phenomena.
For example, Poisson manifolds, symplectic orbifolds, Hamiltonian \(G\) spaces, and integrable systems, all give rise to natural examples of D-Lie groupoids.
We will also introduce the infinitesimal version of a D-Lie groupoid, a \emph{D-Lie algebroid}.

In Section~\ref{section:principalgbundles} we consider principal D-Lie groupoids bundles which allows us to introduced the notions of Morita equivalence, morphism and weak equivalence of D-Lie groupoids.
This provides the answers to (Q1) and (Q2) above.

In Section~\ref{section:stacks} we discuss stacks over $\DMan$. If $\G$ is a D-Lie groupoid we define $\B\G$ as the fibered category over $\DMan$ consisting of all principal D-Lie $\G$-bundles.
We show that the functor $\B$ relates a D-Lie groupoid with a (presentable) stack over $\DMan$, similar to the usual presentation of differentiable stack by Lie groupoids.
Since symplectic groupoids are examples of D-Lie groupoids, this provides an answer to (Q3). Moreover, we prove the following result which provides the bridge with the notion of symplectic Morita equivalence that one finds in the literature:

\begin{theorem}\label{thm:main1}
Let $\G$ and $\H$ be symplectic groupoids.
The following are equivalent:
\begin{enumerate}[(1)]
\item $\G$ and $\H$ are symplectically Morita equivalent.
\item $\B\G$ is isomorphic to $\B\H$.
\item There exists a principal $(G,\H)$-bibundle of D-Lie groupoids.
\item There exists a pre-symplectic groupoid $\G'$ and a pair of weak equivalences of D-Lie groupoids $\G \from \G' \to \H$.
\end{enumerate}
\end{theorem}
%

Our results show that stacks over $\DMan$ provide an appropriate categorical formalism for Poisson geometry.
For example, the `separated' stacks associated to D-Lie groupoids are closely related to the Poisson manifolds of compact type of Crainic, Fernandes, and Martinez-Torres \cites{PMCT1,PMCT2}.
Moreover, the forgetful functor $\DMan \to \Man$ extends naturally to a functor from (presentable) stacks over $\DMan$ to (presentable) stacks over $\Man$.
Hence, one may think of a presentable stack over $\DMan$ as a singular object of $\Man$ equipped with additional geometry.

Moving now to non-integrable Poisson manifolds, if  $A\to M$ is any Lie algebroid we will denote by $\Sigma(A)$ the corresponding source 1-connected topological groupoid, consisting of $A$-paths modulo $A$-homotopies.
Recall that $A$ is an integrable Lie algebroid if and only if $\Sigma(A)$ is a Lie groupoid, and that any morphism $F: A \to B$ of Lie algebroids integrates to a groupoid morphism $\F: \Sigma(A) \to \Sigma(B)$ (see, e.g.,~\cite{Cint}). In Section~\ref{section:infweakequivalences} we give the following infinitesimal characterization of weak equivalences:
\begin{theorem}\label{thm:infweakequiv}
Let $F: A \to B$ be a morphism of integrable Lie algebroids.
The corresponding morphism $\F: \Sigma(A) \to \Sigma(B)$ is a weak equivalence if and only if the following hold for all $x \in M$:
\begin{enumerate}[(a)]
\item $F$ induces a homeomorphism of orbit spaces;
\item $F$ is transverse;
\item the map of isotropy algebras $F_x: \g_x \to \g_{f(x)}$ is an isomorphism;
\item the map of monodromy groups $\N_x(A) \to \N_x(B)$ is an isomorphism;
\item the map of fundamental groups $\pi_1(\O_x) \to \pi_1(\O_{f(x)})$ is an isomorphism.
\end{enumerate}
\end{theorem}

Notice that conditions (a-e) can be stated purely in terms of the Lie algebroid morphism $F: A \to B$, so this result suggests a definition of infinitesimal weak equivalences of Lie algebroids, and hence also of \emph{any} D-Lie algebroids.
Similarly, by treating a Dirac structure as a D-Lie algebroid, we are able to propose a definition of weak equivalence of Dirac manifolds: a morphism of their associated \mbox{D-Lie} algebroids whose underlying Lie algebroid morphism satisfies (a-e). This definition leads to our next result:
\begin{theorem}\label{thm:main2}
Suppose $M$ and $N$ are integrable Poisson manifolds.
Let $\Sigma(M)$ and $\Sigma(N)$ be their source simply connected integrations.
Then $\Sigma(M)$ and $\Sigma(N)$ are Morita equivalent if and only if there exists a Dirac manifold $X$ and pair of weak equivalences $M \from X \to N$.
\end{theorem}

Since the infinitesimal version of weak equivalence is perfectly well defined for non-integrable manifolds, we have a natural version of Morita equivalence valid also for non-integrable Poisson manifolds: two Poisson manifolds $M$ and $N$ are \emph{infinitesimally Morita equivalent} if there exists Dirac manifolds $ \{ X^i \}$ and a chain of weak equivalences
\[ M \from X^1 \to X^2 \from \cdots \to X^{n-1} \from X^n \to N \, . \]
When $M$ and $N$ are integrable, this coincides with ordinary Morita equivalence of Poisson manifolds.

In~\cite{BStacky}, Burzstyn, Noseda, and Zhu characterized principal bundles for \emph{stacky} groupoids, which are the geometric objects that `integrate' non-integrable algebroids.
This also gives rise to a natural definition of Morita equivalence of non-integrable algebroids, which we call BNZ equivalence:  two algebroids are BNZ equivalent if and only if there exists a principal bibundle of their stacky (holonomy) integrations.
We conjecture that infinitesimal Morita equivalence and BNZ equivalence coincide.
This will be discussed in a upcoming paper.

Finally, although the main focus of this paper are Poisson manifolds and their associated symplectic groupoids, our results extend to Dirac structures with a background 3-form and their associated presymplectic groupoids.

\textbf{Acknowledgements} The author would like to thank his thesis advisor Rui Loja Fernandes for his guidance and support throughout the preparation of this document.
The author would also like thank Eugene Lerman and Matias del Hoyo for several comments and discussions related to this work. Finally, he is very grateful for the anonymous referee's constructive feedback.
\section{Preliminaries}\label{section:dman}
Poisson tensors are not well behaved in a category theoretic sense.
For example, given a submersion $f: N \to M$, and a Poisson structure on $M$, we generally cannot pull-back the Poisson structure on $M$ to $N$.
We can interpret this problem geometrically:
If we think of a Poisson manifold, $M$ as a singular foliation by symplectic manifolds, then we can pull back this foliation along any submersion.
However, the leaves of the resulting foliation will no longer be symplectic manifolds but \emph{pre}symplectic manifolds: they form a Dirac structure.
\subsection{Dirac structures}
Intuitively, Dirac manifolds are manifolds equipped with a singular foliation by \emph{pre}symplectic manifolds.
We will provide a brief overview of Dirac structures and Lie groupoids in order to establish our notation conventions.
For a more detailed discussion of Dirac structures see~\cite{BDiracintro}\cite{MMbook}.

Given a smooth manifold, $M$, we call $\T M = TM \oplus T^* M$ the \emph{generalized tangent bundle} of $M$.
Elements of $\T_x M$ will be denoted by pairs $v \oplus \eta$ where $v \in T_x M$ and $\eta \in T_x^* M$.
Generally we will use $\eta$ and $\zeta$ to denote 1-forms or co-vectors while $\alpha$ and $\beta$ will denote 2-forms.

The generalized tangent bundle comes with a natural symmetric product and a bracket operation.
A Dirac structure on $M$ is an involutive, maximally isotropic, subbundle $L \le \T M$.
On such subbundles, the $\T M$-bracket makes $L$ into a Lie algebroid.
We will denote the resulting Lie bracket on sections of $L$ by square brackets ${[\cdot,\cdot]}_L$.
Furthermore, $\rho_L$ will denote the anchor map $v \oplus \eta \mapsto v$. The pair $(M,L)$ is called a \emph{Dirac manifold}.
The orbits (or leaves) of $L$, are the maximal submanifolds integrating the singular distribution $\rho_L(L)$. Every orbit $\O$ of a Dirac structure comes with a closed 2-form, denoted $\omega^{\mathcal{O}}$.

Poisson manifolds are a special case of Dirac manifolds where $L$ is given by the graph of a Poisson bivector.
Given a manifold $M$ equipped with a Poisson bivector $\pi$, we will denote the associated Dirac structure by $L_\pi$.
Similarly, the graph of any closed 2-form $\omega$ defines a Dirac structure which we denote by $L_\omega$.
\subsection{The category of Dirac manifolds}
Given a smooth map $f: N \to M$ and a Dirac structure $L_M$, we can define a pullback operation where
\begin{equation}\label{eqn:diracpullback}
{(f^* L_M)}_x := \{  w \oplus f^* \eta : \dif f (w) \oplus \eta \in {(L_N)}_{f(x)} \} \, .
\end{equation}
Geometrically this corresponds to pulling back the foliation on $M$ and the associated leafwise 2-forms.
Unfortunately this construction does not always produce a Dirac structure: $f^* L_M$ may not be a smooth vector  subbundle.
However, if $f: N \to M$ is transverse to the orbits of $L_M$ then $f^* L_M$ is always a well defined Dirac structure~\cite{BDiracintro}.
In particular, Dirac structures can always be pulled back along submersions.

Now suppose $(M,L_M)$ is a Dirac manifold and $\beta \in \Omega^2(M)$ is a closed 2-form on $M$.
The \emph{gauge transform} of $L_M$ by $\beta$ is the Dirac structure  $L_M + \beta$ defined at each $x \in M$ by:
\begin{equation}\label{eqn:gaugetransform}
{(L_M + \beta)}_x := \{ v \oplus (\eta + \iota_v(\beta)) \in \mathbb{T}M_x : v \oplus \alpha \in L_M.  \}
\end{equation}
Geometrically, this corresponds to adding to each leafwise 2-form $\omega^{\mathcal{O}}$ the pullback of $\beta$ to $\mathcal{O}$.
We can now introduce our main category of study.
\begin{definition}
The \emph{category of Dirac Manifolds} $\DMan$ is defined as follows:
\begin{itemize}
\item the \emph{objects} of $\DMan$ are Dirac manifolds $(M,L_M)$;
\item the \emph{morphisms} of $\DMan$ are pairs $(f,\beta): (N , L_N) \to (M,L_M)$ where $f$ is smooth map such that $f^* L_M$ is a well defined Dirac structure on $N$ and $\beta \in \Omega^2(N)$ is a 2-form such that
\begin{equation}\label{eqn:diracmorphism}
f^*L_M = L_N + \beta.
\end{equation}
\end{itemize}
\end{definition}
Composition in this category is by the rule $(f , \beta) \circ (g, \alpha) = (f \circ g, \, g^* \beta + \alpha)$.
\begin{example}[Gauge Transformations]
Suppose $(M,L_M)$ is a Dirac manifold and $\beta$ is a closed 2-form on $M$, then $(\Id, \beta): (M,L_M) \to (M,L_M+\beta)$ is a morphism in $\DMan$.
We call such morphisms \emph{gauge transformations}.
\end{example}
\begin{example}[Smooth Maps]
Let $M$ and $N$ be any smooth manifolds.
Then $(M, TM)$ and $(N,TN)$ are Dirac manifolds.
For any smooth map $f: M \to N$ we have that $(f,0) : (M, TM) \to (N, TN)$ is a morphism in $\DMan$.
\end{example}
\begin{example}[Symplectic Leaves]
Suppose $(M, L_M)$ is a manifold and $L_M$ is the graph of a Poisson bivector.
Any orbit $\O$ of $M$ has an associated symplectic form $\omega^\O$ and the immersion $i:\O \to M$ satisfies $i^* L_M = L_{\omega^\O}$.
\end{example}
To simplify our notation we will sometimes denote a morphism $(f,\beta)$ in $\DMan$ by $f$ alone and the 2-form $\beta$ will be called the \emph{gauge part} of $f$.
Similarly we may sometimes denote a Dirac manifold $(M,L_M)$ by $M$ alone.
The notation $L_M$ will always denote the Dirac structure of $M$.
Lastly, if we say a morphism in $\DMan$ is a \emph{submersion} we mean that the underlying smooth map is a submersion.

The category $\DMan$ comes with a natural functor $\Pr_1 : \DMan \to \Man$ by projection to the first factor of each pair.
This functor is split by a fully faithful functor $\mathbf{i}: \Man \to \DMan$ which takes any manifold $M$ to the Dirac manifold $(M, TM)$ and any smooth map $f$ to $(f,0)$.

We can characterize commutative diagrams in $\DMan$ by considering the associated diagram in $\Man$ together with a \emph{gauge equation}.
For example, suppose we are given a triangle $T$ of morphisms in $\DMan$ as per (\ref{eqn:triangle}).
\begin{equation}\label{eqn:triangle}
T =
\begin{tikzcd}
 & M_2 \arrow[rd, "f_2"] &  \\
M_1 \arrow[ru, "f_1"] \arrow[rr,"f_3"] & & M_3
\end{tikzcd}
\end{equation}
The \emph{gauge part} of $T$ is the equation $\beta_1 + f_1^*\beta_2 = \beta_3$ (here $\beta_i$ is the gauge part of $f_i$).
More generally, any diagram $D$ in $\DMan$ comes with a set of gauge equations coming from each triangle in $D$.
We can see immediately that, $D$ is a commuting diagram if and only if $\Pr_1(D)$ commutes in $\Man$ and each gauge equation holds.

Suppose we are given two morphisms $(f,\beta): M \to X$ and $(g,\alpha): N \to X$ in $\DMan$ such that the manifold $M \times_X N$ exists.
Then the \emph{fiber product} is defined to be $M \times_X N$ where
\begin{equation}\label{eqn:fiberproduct}
L_{M \times_X N} := {(f \circ \pr_1)}^* L_X - \pr_1^* \beta - \pr_2^* \alpha .
\end{equation}
Such a fiber product fits into a corresponding pullback square in $\DMan$:
\[
\begin{tikzcd}
M \times_X N \arrow[r, "\pr_2"] \arrow[d, "\pr_1"]  & N  \arrow[d,"g"] \\
M \arrow[r, "f"] & X
\end{tikzcd}
\]
We take the gauge parts of $\pr_1$ and $\pr_2$ to be $\pr_2^* \alpha$ and $\pr_1^* \beta$ respectively.
Observe that such a fiber product always exists if either $f$ or $g$ is a submersion.

Fiber products in $\DMan$ still satisfy the same universal property.
Suppose we have the following diagram in $\DMan$:
\[
\begin{tikzcd}
Y  \arrow[dr, "k", dashrightarrow] \arrow[drr, "h_2", bend left] \arrow[ddr, "h_1" swap, bend right] & & \\
  & M \times_X N \arrow[r, "\pr_2"] \arrow[d, "\pr_1"]  & N  \arrow[d,"g"] \\
  & M \arrow[r, "f"] & X
\end{tikzcd}
\]
Let $\eta_1$ and $\eta_2$ be the gauge parts of $h_1$ and $h_2$, respectively.
Then the gauge equation arising from the outermost square is
\begin{equation}\label{eqn:universalprop}
h_1^* \beta + \eta_1 = h_2^* \alpha + \eta_2 \, .
\end{equation}
We already know that there is a unique smooth map $k:Y \to M \times_X N$ which makes this diagram commute.
We can define the gauge part of $k$, call it $\kappa$, one of two ways:
\[ \kappa + k^* \pr_1^* \beta = \eta_2 \quad \mbox{ or equivalently} \quad  \kappa + k^* \pr_2^* \beta = \eta_1 \, . \]
In the presence of (\ref{eqn:universalprop}), these definitions are equivalent.
They must hold in order for the diagram to commute since they represent the gauge equations of the top and left triangles created by inserting $k: Y \to M \times_X N$ into the diagram above.
Hence, $(k,\kappa)$ is the unique morphism which completes the diagram in $\DMan$.
\subsection{Groupoids and bibundles}
We will now briefly establish our notation for Lie groupoids and their principal bibundles.
A more detailed exposition on the subject can be found in~\cite{BPic,lecturesonint}.

We will denote a Lie groupoid by $\G \rightrightarrows M$, so $\G$ and $M$ are the manifolds of arrows and objects.
We will denote by $\s,\t,\u,\m,\i$ the source, target, unit, multiplication and inverse maps, respectively.
These maps satisfy the appropriate groupoid axioms and the source and target maps are submersions.
For each natural number $n>1$ we denote by $\G^{(n)}$ the manifold of $n$-tuples of composable arrows.

Given a Lie groupoid $\G \rightrightarrows M$ and a map $\t^P:P\to M$ a \emph{left $\G$-action} is specified by a map $\m_L: \G \times_M P \to P$ satisfying the usual axioms of an action.
If write $g \cdot p=\m_L(g,p)$ these can be written as:
\[ g_1\cdot (g_2\cdot p)=(g_1 \cdot g_2)\cdot p,\quad \u(\t^P(p))\cdot p=p.  \]
We say that $P$ is a \emph{left $\G$-bundle over $N$} is there is a submersion $\s^P:P\to N$ which is $\G$-invariant, i.e.:
\[ \s^P(g\cdot p)=\s^P(p),\quad \forall (g,p)\in \G \times_M P.  \]
The left $\G$-bundle is called \emph{principal} if the map $\G \times_M P \to P \times_N P$, $(g,p) \mapsto(g \cdot p, p)$, is a diffeomorphism.
We have similar notations for a \emph{right $\G$-action}.

Given Lie groupoids $\G \rightrightarrows M$ and $\H \rightrightarrows N$, a \emph{$(\G, \H)$-bibundle} is a manifold $P$ which is both a left $\G$-bundle over $N$ and a right $\H$-bundle over $M$, such that the two actions commute.
If $P$ is both left and right principal then $P$ is a called a \emph{principal $(\G, \H)$-bibundle}. Principal $(\G,\H)$-bundles are also known as \emph{Morita equivalences}.

Given a $(\G_1,\G_2)$-bibundle $P_1$ and a $(\G_2,\G_3)$-bibundle $P_2$, we can construct a tensor product $P_1 \otimes_{\G_2} P_2$ where
\[
P_1 \otimes P_2 := \frac{P_1 \times_{M_2} P_2}{(p_1 \cdot g_2, p_2) \sim (p_1, g_2 \cdot p_2)} \, .
\]
This composition is associative (up to bibundle isomorphism).
If $P$ is a $(\G,\H)$-bibundle then $\G \otimes P \isom P$ and $P \otimes \H \isom P$.
Lastly, $P$ is invertible (i.e.\ there exists $P\inv$ such that $P\inv \otimes P \isom \H$ and $P \otimes P\inv \isom \G$) if and only if $P$ is principal.
\begin{example}
Given a homomorphism of Lie groupoids $F: \H \to \G$ covering $f: M \to N$ then let $P_F$ be the left principal $(\G,\H)$-bibundle constructed as follows:

As a manifold $P_F = \H \times_{\s, f} N$.
The left and right actions on $P_F$ are given by:
\[ \m_L(g',(g,x)) = (g' g, x), \qquad \m_R((g , x), h) = (g F(h),\s(h)) \, . \]
The assignment $F \mapsto P_F$ is functorial in the sense that it satisfies
\begin{equation}
P_{F \circ G} \isom P_F \otimes P_G \, .
\end{equation}
\end{example}

Given a Lie algebroid $(A, {[ \cdot , \cdot ]}_A , \rho_A)$ we will use $\Sigma(A)$ to denote the canonical source simply connected integration of $A$.
When $A = L_M$ is a Dirac structure, the notation $\Sigma(L_M)$ denotes the canonical source simply connected pre-symplectic groupoid integrating $L_M$
(see~\cite{integrationoftwisteddiracbrackets} for a treatment of this construction).
\section{Groupoids in \texorpdfstring{$\DMan$}{}}\label{section:dlie}
\begin{definition}
In brief, a \emph{D-Lie groupoid} is a groupoid object internal to the category $\DMan$.
Hence, it is a pair of objects $\G$ and $M$ in $\DMan$ together with morphisms $\s,\t,\u,\m,\i$ satisfying the groupoid axioms (we also require that $\s$ and $\t$ be submersions).
\end{definition}

The notion of a D-Lie groupoid should not be confused with the notion of a multiplicative Dirac structure on a groupoid, the so-called Dirac groupoids  \cites{DiracLieGroups,Ortiz}, which include Poisson-Lie groups and Poisson groupoids as special cases.
In general, Dirac groupoids \emph{do not} (in any obvious way) provide examples of D-Lie groupoids.

Each groupoid axiom can be interpreted as a diagram in $\DMan$.
Hence, if we are supplied with Dirac manifolds $\G$ and $M$ and morphisms $\s, \t,\u, \m, \i$ then the resulting data is a D-Lie groupoid if and only if the projection under $\Pr_1$ is a Lie groupoid and the associated gauge equations of each groupoid axiom holds.
A \emph{homomorphism} of D-Lie groupoids is defined in the natural way:
it is a pair of morphisms, $F: \G \to \H$ and $f: M \to N$, in $\DMan$, such that $F$ is compatible with the source, target and multiplication maps.

The next lemma allows us to give a more geometric characterization of D-Lie groupoids.
\begin{lemma}\label{lemma:dliedata}
Let $\G$ be a Lie groupoid over $M$ and $L_M$ be a Dirac structure on $M$.
Suppose $\tau$ and $\sigma$ are a pair of closed 2-forms on $\G$ such that
\begin{enumerate}[(i)]
\item $\t^* L_M = \s^* L_M + (\tau - \sigma)$ and
\item $\tau-\sigma$ is multiplicative.
\end{enumerate}
Then there is a unique D-Lie groupoid whose underlying Lie groupoid is $\G$ and has source and target morphisms $(\s,\sigma)$ and $(\t,\tau)$.
\end{lemma}
A detailed proof of this lemma can be found in Appendix~\ref{appendix1}.
Still, let us outline the main idea.
\begin{proof}[Outline]
The key observation is that each groupoid axiom can be interpreted as a commutative diagram.
Therefore, each groupoid axiom has an associated gauge equation, which is an equation of 2-forms involving the gauge parts of each structure maps.
Examining these equations reveals that the gauge part of each structure map can be written entirely in terms of $\tau$ and $\sigma$.
For instance, gauge part of $\s \circ \m = \s \circ \pr_2$ yields the equation
\begin{equation}\label{eqn:mudef}
\mu = \pr_1^* \sigma - \m^* \sigma + \pr_2^* \sigma
\end{equation}
(here $\mu$ is the gauge part of $\m$).
Observe that given such a $\sigma$ and $\tau$ that $L_\G := \t^* L_M - \tau$ is a well defined Dirac structure on $\G$.
Furthermore, condition (i) is equivalent to saying that $(\t,\tau)$ and $(\s,\sigma)$ are morphisms of Dirac manifolds.
To construct the D-Lie groupoid, we must produce the 2-forms corresponding to the remaining structure maps.
What we do then is take equations such as (\ref{eqn:mudef}) to be definitions.
We are left to check that that the assumption that $\tau-\sigma$ is multiplicative suffices to ensure that this produces a well defined D-Lie groupoid.
This amounts to showing that multiplicativity of $\tau-\sigma$ implies the gauge part of every groupoid axiom.
This shows that from such a $\sigma$ and $\tau$ we can define a D-Lie groupoid.
Since the groupoid axioms imply that the gauge part of each structure map depends on $\sigma$ and $\tau$, uniqueness is immediate.
\qed\end{proof}

From now on we will call the pair $(\tau,\sigma)$ the \emph{gauge pair} of $\G$.
The 2-form $\Omega := \tau - \sigma$ will be called the \emph{characteristic form} of $\G$.
One way to think of condition (i) is that $\Omega$ measures the degree to which $\t^* L_M \neq \s^* L_M$.
Condition (ii) says that this failure must be up to a multiplicative gauge transformation.

If a D-Lie groupoid $\G$ has a gauge pair of the form $(\tau,0)$ then $\Omega = \tau$ and we call $\G$ \emph{target aligned}.
It turns out that, up to isomorphism, it suffices to consider target aligned D-Lie groupoids.
\begin{lemma}\label{lemma:targetalign}
Let $\G \rightrightarrows M$ be a D-Lie groupoid over $M$ with gauge pair $(\tau,\sigma)$.
Then the pair $(\tau - \sigma,0)$ determines a target aligned D-Lie groupoid and the gauge transformation:
\[ (\Id_\G, \sigma): ( \G, L_\G ) \to (\G , L_\G + \sigma), \, \]
is an isomorphism.
\end{lemma}
\begin{proof}
For ease of notation, let $\G'$ denote the Dirac manifold $(\G, L_\G + \sigma)$.
By Lemma~\ref{lemma:dliedata}, $\G'$ is a D-Lie groupoid with gauge pair $(\tau-\sigma,0)$.
The morphism $(\Id_\G,\sigma)$ is clearly an isomorphism of Dirac manifolds.
It only remains to check that $(\Id_\G, \sigma)$ is a homomorphism of D-Lie groupoids.
We first verify that $(\Id_\G, \sigma)$ is compatible with the source and target maps, i.e.
\begin{equation} (\s, 0) \circ (\Id, \sigma) = (\Id, 0 ) \circ (\s, \sigma) \quad \mbox{and} \quad (\t, \Omega) \circ (\Id, \sigma) = (\Id, 0 ) \circ (\t, \tau) \, .
\end{equation}
The above equalities are clear from the definition of composition in $\DMan$.
To see that $(\Id_\G, \sigma)$ is compatible with the multiplication, we must check that:
\[
\begin{tikzcd}[column sep=6.5em]
\G \times_M \G \arrow[d, "{(\m, \mu)}"] \arrow[r, "{( \Id, \pr_1^* \sigma + \pr_2^* \sigma )}"] & \G' \times_M \G' \arrow[d, "{(\m, 0)}"] \\
\G \arrow[r, "{(\Id, \sigma )}"] & \G' \,
\end{tikzcd}
\]
commutes.
The gauge equation associated to this diagram is $m^* \sigma + \mu = \pr_1^* \sigma + \pr_2^* \sigma$ which holds by (\ref{eqn:mudef}).
\qed\end{proof}
The next two examples give important classes of D-Lie groupoids which illustrate the scope of this notion.
\begin{example}[Symplectic Groupoids]
Given $(\G, \Omega)$, a symplectic groupoid integrating a Poisson manifold $(M,L_\pi)$, then $t^* L_\pi = s^*L_\pi + \Omega$.
Therefore, a symplectic groupoid is the same as a target aligned D-Lie groupoid with non-degenerate characteristic form.
\end{example}
\begin{example}[Symplectic orbifolds]
Let $\G \rightrightarrows M$ be an {\'e}tale Lie groupoid and suppose $\omega$ is a symplectic form on $M$.
Suppose further that $\t^* \omega - \s^* \omega = 0$.
When $\G$ is proper, it can be thought of as the presentation of a (possibly non-effective) symplectic orbifold.
By thinking of $M$ as a Dirac manifold, then $\G$ can also be thought of as a D-Lie groupoid with characteristic form $0$.
\end{example}
The two preceding examples can be thought of as extreme cases of D-Lie groupoids.
In the first case, the characteristic 2-form on the space of arrows is non-degenerate.
In the second case, the characteristic 2-form vanishes.
In general, a typical D-Lie groupoid is something in between a symplectic groupoid and a groupoid equipped with a multiplicative Dirac structure on the space of objects.
\subsection{D-Lie groupoid morphisms}\label{subsection:groupoidmorphisms}
We will now take a closer look at homomorphisms of D-Lie groupoids.
Throughout, $\G$ and $\H$ are D-Lie groupoids over $M$ and $N$ respectively.
Also, $\Omega^\G$ and $\Omega^\H$ will denote their respective characteristic forms.

Lemma~\ref{lemma:dliedata} also has a version for morphism, so we can also interpret morphisms of D-Lie groupoids in terms of the characteristic forms:
\begin{lemma}
Suppose $F: \G \to \H$ is a homomorphism of the underlying {\bf Lie} groupoids covering the smooth map $f: M \to N$.
Let $\beta$ be a closed 2-form such that $(f,\beta)$ is a morphism of Dirac manifolds and suppose that
\begin{equation}\label{eqn:morphism}
F^* \Omega^\H = t^* \beta - s^* \beta + \Omega^\G \, .
\end{equation}
Then there is a unique 2-form $\alpha$ making $(F,\alpha)$ into a D-Lie groupoid morphism covering $(f,\beta)$.
Furthermore, every D-Lie groupoid morphism arises in this way.
\end{lemma}
\begin{proof}
To prove the first part, we just need to supply a suitable $\alpha$.
Let $(\tau^\G, \sigma^\G)$ and $(\tau^\H, \sigma^\H)$ be the gauge pairs of $\G$ and $\H$.
Now define $\alpha$ to be the unique 2-form so that
\begin{equation}\label{eqn:scompatible}
F^* \sigma^\H + \alpha = \s^* \beta + \sigma^\G .
\end{equation}
If $\alpha$ does the job, it is certainly the unique one since (\ref{eqn:scompatible}) is the gauge part of compatibility with the source.
The gauge part of compatibility with the target is the analogous equation:
\begin{equation}\label{eqn:tcompatible}
F^* \tau^\H + \alpha = \s^* \beta + \tau^\G.
\end{equation}
Which follows from combining (\ref{eqn:morphism}) and (\ref{eqn:scompatible}).
We leave it to the reader to verify that the gauge part of compatibility with multiplication follows from (\ref{eqn:scompatible}) and (\ref{eqn:tcompatible}).
Hence $(F,\alpha)$ is a morphism of D-Lie groupoids.

Certainly all morphism of D-Lie groupoids will be of this form since (\ref{eqn:morphism}) can be obtained by subtracting (\ref{eqn:scompatible}) from (\ref{eqn:tcompatible}).
\qed~\end{proof}
\begin{example}[Symplectic Groupoids]
Suppose $\G \rightrightarrows M$ and $\H \rightrightarrows N$ are symplectic groupoids.
If we think of $\G$ and $\H$ as target aligned D-Lie groupoids then a morphism consists of a homomorphism of Lie groupoids $F: \G \to \H$ together with a closed 2-form $\beta \in \Omega^2(M)$ such that (\ref{eqn:morphism}) holds.
\end{example}
\subsection{D-Lie algebroids}
There is an infinitesimal version of D-Lie groupoids.
Recall from~\cite{BImf} that given a closed multiplicative form on a Lie groupoid, there is a corresponding \emph{infinitesimal multiplicative form} on the corresponding algebroid.
In the case of closed 2-forms an infinitesimal multiplicative form is a bundle map
\[ \rho_A^*: A \to T^*M  \]
which satisfies a compatibility condition with the bracket on $A$.
Note that we have used the $^*$ notation to emphasize that the map takes values in the cotangent bundle.
We do not mean that $\rho_A^*$ is the linear dual of the anchor.
Here, $A$ is the Lie algebroid of the Lie groupoid in consideration.

Suppose $\Omega$ is a closed multiplicative form on a Lie groupoid.
From~\cite{BImf}, we can define $\rho_A^*$ in the following way:
\begin{equation}\label{eqn:IMFdef}
\rho_A^*(v) := \eta  \quad \Leftrightarrow \quad  t^* \eta = \Omega^\flat(v) \, .
\end{equation}
Furthermore,~\cite{BImf} tells us that $\rho_A^*(v)$ is compatible with the bracket, i.e.
\begin{equation}\label{eqn:IMFbracket}
\rho_A^*([V,W]) = \L_V(\rho_A^*(W)) - {(\dif \rho_A^*(V))}^\flat(W) \, .
\end{equation}
We now proceed to a simple lemma, which will motivate our definition of D-Lie algebroid.

\begin{lemma}\label{lemma:dliealg}
Suppose $\G \rightrightarrows M$ is a D-Lie groupoid with characteristic form $\Omega$.
Let $A$ be the corresponding algebroid and $\rho_A^*$ be the associated infinitesimal multiplicative form and $\rho_A$ be the anchor map.
Then $\rho_A(v) \oplus \rho_A^*(v) \in L_M$ for all $v \in A$ and $(\rho \oplus \rho_A^*): A \to L_M$ is a Lie algebroid homomorphism.
\end{lemma}
\begin{proof}
We have two things to show.
First, the claim that, for any $v \in A$, $\rho_A(v) \oplus \rho_A^*(v) \in L_M$.
Let $v \in A := \ker \dif \s|_M$ and suppose $\eta = \rho_A^*(v)$.
By the definition of the pullback, we know that $v \oplus 0 \in \s^* L_M$.
Consequently, $v \oplus \Omega^\flat(v) \in \t^* L_M$ by Lemma~\ref{lemma:dliedata}\emph{(i)}.
Hence, by the definition of $\rho_A^*$, we have that $v \oplus \eta \in \t^* L_M$.
Therefore, we can conclude that $\rho_A(v) \oplus \rho_A^*(v) \in L_M$ (again by the definition of the pullback).

Now for the second part.
Recall the definition
\footnote{Note that some references may give a slightly different definition of the Courant bracket.
In this case, for convenience we are adopting the version used in~\cite{integrationoftwisteddiracbrackets}.
When restricted to any given Dirac structure, they turn out to be equal.
} of the Courant bracket on $TM \oplus T^*M$, which plays the roll of the Lie bracket for $L_M$:
\begin{equation}\label{eqn:courantbracket}
{[V \oplus \eta , W \oplus \zeta ]}_{L_M} := [V,W] \oplus \L_V(\zeta) - {(\dif \eta)}^\flat(W) \, .
\end{equation}
We need to show that
\[ {[\rho_A(V) \oplus \rho_A^*(V), \rho_A(W) \oplus \rho_A^*(W) ]}_{L_M} = \rho_A({[V,W]}_A) \oplus \rho_A^*({[V,W]}_A) \, . \]
This is clearly true for the $TM$ component, since $\rho: A \to TM$ is compatible with the standard Lie bracket.
For the $T^*M$ component, the result follows immediately from (\ref{eqn:IMFbracket}).
\qed\end{proof}
\begin{definition}
A \emph{D-Lie algebroid} is a Lie algebroid $(A, {[ \cdot , \cdot ]}_A, \rho)$ over a Dirac manifold $(M,L)$ together with a Lie algebroid homomorphism $\til{\rho}: A \to L_M$
\end{definition}

We can recover an infinitesimal multiplicative form from this definition by taking $\rho_A^*$ to be the $T^*M$ component of $\til\rho$.
We say that a D-Lie algebroid is \emph{integrable} if there exists a D-Lie groupoid $\G$ whose corresponding infinitesimal multiplicative form is $\rho_A^*$.
When $\G$ is source simply connected and target aligned we say that $\G$ is the canonical integration.
\begin{example}[Poisson Manifolds]
Let $(\G,\Omega)$ be a symplectic groupoid over a Poisson manifold $(M,\pi)$.
Then let $\til\rho: A \to T^*M$ be the standard identification of the algebroid of $\G$ with the cotangent bundle of $M$.
In this way, we can think of $A$ as a D-Lie algebroid.
\end{example}
\begin{example}[Dirac Structures]
Let $L_M$ be any Dirac structure.
Then if we take $\til\rho: L_M \to L_M$ to be the identity morphism, we can think of $L_M$ as a D-Lie algebroid.
\end{example}
\begin{example}[Trivial algebroids]
Let $A$ be a rank zero vector bundle, thought of as a trivial algebroid.
Then for any Dirac structure $L_M$ on the base of $A$, we can take $\til\rho: A \to L_M$ to be the zero map.
\end{example}
The last example illustrates the interesting fact that integrability of $L_M$ is neither a necessary nor a sufficient condition for integrablity of the D-Lie algebroid.
\begin{definition}
Suppose $(A, \rho_A^*)$ and $(B, \rho_B^*)$ are D-Lie algebroids.
A \emph{morphism} of D-Lie algebroids is a Lie algebroid morphism $F: A \to B$ which cover a morphism of Dirac manifolds $(f,\beta): M \to N$ such that
\[ f^* \circ \rho_B^* \circ F = \rho_A^* + \beta^\flat \, . \]
\end{definition}
The left side of the equation above is the \emph{pullback} of the IM 2-form $\rho_B^*$ along $F$.
The right side is the infinitesimal form of the gauge transformation $\Omega + \t^* \beta - \s^* \beta$.
Hence, this is just the infinitesimal version of (\ref{eqn:morphism}).

\section{Principal bundles and Morita equivalence}\label{section:principalgbundles}
In this section, we will define $G$-bundles and $(\G,\H)$-bibundles in the setting of D-Lie groupoids.
This gives rise to a notion of Morita equivalence of D-Lie groupoids which generalizes Morita equivalence for symplectic groupoids.
\subsection{Principal \texorpdfstring{$\G$}{} bundles}
\begin{definition}
Let $\G \rightrightarrows M$ be a D-Lie groupoid.
A \emph{left $\G$-bundle} over $N$ is a $\G$-bundle internal to $\DMan$, i.e., an ordinary $\G$-bundle $\s^P:P\to N$, where $P$ and $N$  are Dirac manifolds and all the structure maps $\s^P:P \to N$, $\t^P:P \to M$ and $\m_L:\G \times_M P \to P$ are $\DMan$-morphisms.
We say that $P$ is principal if the induced morphism $\G \times_M P \to P \times_N P$ is an isomorphism.
\end{definition}

The reader should note that, by our construction of fiber products in $\DMan$, a $\G$-bundle for a D-Lie groupoid is principal if and only if the underlying action of a Lie groupoid is principal.

The morphisms $\s^P, \t^P$ and $\m_L^P$ come with gauge parts $\sigma^P, \tau^P$ and $\mu_L^P$.
The equation associated to $\s^P \circ \m_L(g,p) = \s^\P(p)$ is
\[ \mu_L + \m^* \sigma^P = \pr_1^* \sigma^\G + \pr_2^* \sigma^P \, . \]
Similarly, the gauge equation of $\t^P \circ \m_L(g,p) = \t(g)$ is
\[ \mu_L + \m^* \tau^P = \pr_1^* \tau^\G + \pr_2^* \tau^P \, . \]
Therefore, when defining a principal $\G$-bundle it suffices to specify the 2-forms $\sigma^P$ and $\tau^P$.
As with D-Lie groupoids, we say the characteristic 2-form of $P$ is $\Omega^P := \tau^P - \sigma^P$.
Let $\Omega^\G$ be the characteristic 2-form of $\G$.
Then
\[ \m_L^* \Omega^P = \pr_1^* \Omega^\G + \pr_2^* \Omega^P \, . \]
That is, the closed 2-form $\Omega^P$ is left multiplicative.
When $\sigma^P = 0$ we call $P$ \emph{target aligned}.
By a similar argument as in Lemma~\ref{lemma:targetalign}, every principal $\G$-bundle is canonically isomorphic to a target aligned principal $\G$-bundle.

The next three lemmas are important technical results about principal $\G$-bundles that will be needed later in Section~\ref{section:stacks}.
In brief, the first says that the standard construction of a pullback $\G$-bundle still works in our setting.
The second lemma implies that the pullback construction is unique up to a unique isomorphism.
The last lemma says that principal $\G$-bundles of D-Lie groupoids satisfy a property known as `descent.'
\begin{lemma}\label{lemma:gbundlepullback}
Let $\G$ be a D-Lie groupoid, $f: N_1 \to N_2$ a morphism in $\DMan$ and $P$ a principal $\G$-bundle over $N_2$.
Then there exists a principal $\G$-bundle $Q$ over $N_1$ and a principal bundle morphism $F: Q \to P$ covering $f$.
\end{lemma}
\begin{proof}
Without loss of generality, we can assume $\G$ and $P$ are target aligned.
We already know that the result holds in $\Man$ for principal $\G$-bundles of Lie groupoids.
Let
\[ Q:= f^* P = P \times_{N_2} N_1 \]
and equip it with the standard structure maps so that $Q$ is a principal $\G$-bundle for the underlying Lie groupoid.
Let $F: Q \to P$ be projection to the first coordinate.
To make $Q$ a principal $\G$-bundle in $\DMan$ we must equip it with a characteristic form.
Let
\[ \Omega^Q := \pr_1^* \Omega^P + \pr_2^* \beta \, .\]
It can easily be verified that $\Omega^Q$ is left multiplicative with respect to the left action of $\G$ on $Q$.
Furthermore the map $F:= \pr_1: Q \to P$ is equivariant with respect to the left action.
\qed\end{proof}
\begin{lemma}\label{lemma:gbundlecartesian}
Suppose $(f_1,\beta_1):N_1 \to N_2$ and $(f_2,\beta_2): N_2 \to N_3$ are morphisms in $\DMan$ and $P_1$, $P_2$ and $P_3$ are principal $\G$-bundles over $N_1$, $N_2$ and $N_3$ respectively.
Furthermore, assume we are given principal bundle morphism $F_2: P_2 \to P_3$ and $G: P_1 \to P_3$ covering $f_2$ and $f_2 \circ f_1$ respectively.
Then there exists a unique principal bundle morphism $F_1: P_1 \to P_2$ covering $f_1$ such that $F_2 \circ F_1 = G$.
\end{lemma}
\begin{proof}
At the level of manifolds, this is a well known property of principal $\G$-bundles for $\G$ a Lie group.
Therefore, to define $F_1$ it suffices to provide the gauge part of $F_1$.
Therefore, let the gauge part of $F_1$ be
\[ \widetilde{\beta} := \sigma^{P_1} - {(\s^{P_1})}^* \beta_1 - {(F_1)}^* \sigma^{P_2} \, . \]
We leave it to the reader to check that $(F_1, \widetilde{\beta})$ is a well defined morphism of principal bundles.
Finally, compatibility with the source maps implies that this choice of gauge part is the only one possible and therefore $(F_1, \widetilde{\beta})$ is unique.
\qed\end{proof}
We will typically denote the principal bundle $Q$ from Lemma~\ref{lemma:gbundlepullback} with the notation $f^* P$ and call it the \emph{pullback bundle} along $f$.
When $f:U \to N$ is an inclusion, we may also denote $f^* P$ by $P|_U$.
\begin{lemma}\label{lemma:gbundleglueing}
\begin{enumerate}[(a)]
\item Suppose $\{ i_a: U_a \to N \}$ is a covering of $N$ in $\DMan$ and $\G$ is a D-Lie groupoid.
Let $P_a \to U_a$ be a collection of principal $\G$-bundles together with morphisms $\phi_{ab}: P_b|_{U_{ab}} \to P_a|_{U_{ab}}$ such that $\phi_{ab} \circ \phi_{bc} = \phi_{ac}$ when restricted to any triple intersection $U_{abc}$.
Then there exists a principal $\G$-bundle $P \to N$ together with morphisms ${\{ \phi_a: P|_{U_a} \to P_a \}}_{a \in A}$ such that $\phi_{ab} \circ \phi_{b} = \phi_a$.
\item Let $P \to N$ and $Q \to N$ be principal $\G$-bundles over $N$ and suppose $\{ i_a : U_a \to N \}$ is a covering of $N$ in $\DMan$.
Suppose we have a collection of morphisms $F_a: P|_{U_a} \to Q|_{U_a}$ covering the identity such that
\[ F_a|_{P|_{U_{ab}}} = F_b|_{Q|_{U_{ab}}} \, . \]
Then there exists a unique morphism $F: P \to Q$ such that $F|_{P|_{U_a}} = F_a$.
\end{enumerate}
\end{lemma}
\begin{proof}
We first prove (a).
As with the previous two lemmas, the result already holds in the smooth category.
Let the manifold $P$ and the smooth maps $\phi_a$ be ones satisfying these properties in the smooth setting.
Without loss of generality we can assume each $P_a$ is target aligned.
We can also assume without loss of generality that the gauge part of each $\iota_a: U_a \to N$ is zero.
To make $P$ into a target aligned principal bundle in $\DMan$, we must specify its characteristic form $\Omega$.
So take
\[ \Omega^P |_{P|_{U_a}}:= \phi_a^* \Omega^{P_a} \, . \]
This is well defined since
\[ \phi_a^* \Omega^{P_a}|_{P|_{U_{ab}}} = \phi_b^* \Omega^{P_b}|_{P|_{U_{ab}}} \, . \]
Which follows from the fact that $\phi_{ab}^* \Omega^{P_b} = \Omega^{P_a}$ together with the fact that the smooth maps $\phi_{ab} \circ \phi_b$ and $\phi_a$ are equal.

We now prove part (b).
Again, assume the maps $i_a: U_a \to N$ have trivial gauge part and $P$ and $Q$ are target aligned.
As with part (a) the result is known to hold in the smooth category.
The smooth map $F: P \to Q$ covers the identity, so we only need to show that $F^* \Omega^Q = \Omega^P$.
However, we know that $F_a^* \Omega^Q = \Omega^P$ when restricted to each $P|_{U_a}$ so the claim follows immediately.
\qed\end{proof}
\subsection{Bibundles}
We can now proceed to tackle bibundles and Morita equivalence in $\DMan$.
Throughout this section $\G$ and $\H$ are D-Lie groupoids over the Dirac manifolds $M$ and $N$ respectively.
\begin{definition}
Suppose $\G$ and $\H$ are D-Lie groupoids.
A \emph{$(\G, \H)$-bibundle} is defined to be a bibundle object internal to the category $\DMan$.
Hence, it is an object $P$ in $\DMan$ together with morphisms $\t^P, \s^P, \m_L,\m_R$ (again in $\DMan$) which satisfy the axioms of commuting left and right actions over $N$ and $M$, respectively.

A bibundle $P$ is said to be \emph{left principal bibundle} if the left action makes $P$ into a left principal $\G$-bundle over $N$.
We define \emph{right principal} similarly.
We call $P$ a \emph{principal bibundle} if $P$ is both left and right principal.
A principal $(\G,\H)$-bibundle is also called a \emph{Morita equivalence} of $\G$ and $\H$.
\end{definition}

Just like $\G$-bundles, a $(\G, \H)$-bibundle $P$ of D-Lie groupoids is determined by the data of the underlying bundle and the gauge part of the source and target maps $\sigma^P$ and $\tau^P$.
The \emph{characteristic form} $\Omega^P$ of $P$ is defined to be $\sigma^P - \tau^P$ as before.
Using the same techniques as before we can show that $\sigma^P$ and $\tau^P$ define a bibundle if and only if $\Omega^P$ is left and right multiplicative.
That is
\[ \m_L^* \Omega^P = \pr_1^* \Omega + \pr_2^* \Omega^P \quad \mbox{ and } \quad \m_R^* \Omega^P = \pr_1^* \Omega^P + \pr_2^* \Omega^\H.  \]
We say that $P$ is \emph{target aligned} if $\sigma^P = 0$.

An equivariant map of $(\G, \H)$-bibundles is a morphism $\Phi: P \to Q$ which commutes with the source and target maps and respects the multiplication.
In terms of the characteristic 2-form the condition on $F: Q \to P$ is just
\[ F^* \Omega^P = \Omega^Q \, . \]
This makes sense when compared to the case of left $\G$-bundles since we can think of any $(\G,\H)$-bibundle morphism as a left $\G$-bundle morphism covering the identity on $N$.
As with D-Lie groupoids, for any bibundle $P$ the gauge transformation $(\Id,\sigma^P): (P,L_P) \to (P,L_P + \sigma^P)$ is an isomorphism of $P$ with a target aligned bibundle.

The next few examples demonstrate how this notion of Morita equivalence of D-Lie groupoid relates to existing definitions of Morita equivalence.
\begin{example}[Morita equivalence of Lie groupoids]
Given a Morita equivalence $P$ of Lie groupoids $\G$ and $\H$, then thinking of $\G$ and $\H$ as D-Lie groupoids with the tangent Dirac structure allows us to view $(P, TP)$ as a Morita equivalence of D-Lie groupoids $(\G, T\G)$ and $(\H, T\H)$.
Furthermore, it is a simple exercise to check that any Morita equivalence of the D-Lie groupoids $(\G, T\G)$ and $(\H, T \H)$ is isomorphic to such a $(P, T P)$.
\end{example}
\begin{example}[Symplectic Morita equivalence]\label{ex:sympl:morita}
Given a symplectic Morita equivalence $(P, \Omega^P)$ of symplectic groupoids $(\G, \Omega^\G)$ and $(\H, \Omega^\H)$, we can think of $(P, \Omega^P)$ as a target aligned Morita equivalence of $\G$ and $\H$ viewed as D-Lie groupoids.
\end{example}

We can improve on the observation from the preceding example.
\begin{proposition}\label{prop:symform}
Suppose $\G$ and $\H$ are symplectic groupoids, i.e.\ target aligned D-Lie groupoids with symplectic characteristic forms.
There is a one-to-one correspondence between symplectic Morita equivalences and target aligned principal $(\G,\H)$-bibundles.
\end{proposition}

\begin{proof}

One direction is just Example~\ref{ex:sympl:morita}.
For the other direction, suppose $P$ is a target aligned principal $(\G,\H)$-bibundle.
We must show that $\Omega^P$ is symplectic at each $p \in P$. So fix $p$ and let $x=\s^P(p)  \in N$ and $y=\t(p)\in M$.
Suppose $e: U \to P$ is a local section of $\s^P$ around $x$ such that $e(x) = p$ and let $\beta := e^* \Omega^P$.
Next define $f:= \t^P \circ e$ and notice that
\begin{align*}
f^* L_M &= e^* {(\t^P)}^* L_M  \\
        &= e^* (L_P + \Omega) \\
        &= e^* L_P + \beta  \\
        &= e^* {(\s^P)}^* L_N + \beta = L_N + \beta \, .
\end{align*}
In other words, $(f,\beta): U \to M$ is a morphism in $\DMan$. Since $P$ is principal, $f$ is transverse to the orbits of $M$.

We will need these facts in a moment, but first we should use the section $e$ to `trivialize' our bibundle.

When $P$ is restricted to $U$, we can identify it with the trivial $\G$-bundle associated to this map.
That is,
\[ P|_U \isom \G \times_{\s,f} U, \text{ with $p$ corresponding to }(\u(f(x)),x) . \]
When $P|_U$ is written in this way, then we can use the left multiplicativity of $\Omega^P$ to see that
\[ \Omega^P = \pr_1^* \Omega^\G + \pr_2^* \beta.  \]
Hence, for any two vectors
\[ (v_i,w_i)\in T_p (\G \times_M N)=\{(v,w)\in T_{\u(f(x))}\G\times T_x N: \dif\s(v)=\dif f(w)\}, \]
we have that
\[ \Omega^P((v_1, w_1), (v_2, w_2)) = \Omega^\G(v_1, v_2) + \beta(w_1,w_2) \, . \]
Now suppose that $(v_1, w_1)$ is in the kernel of $\Omega^P$.
We will show that it must be zero by pairing it with a few careful choices of $(v_2,w_2)$.
First let us see what happens when $(v_2,w_2) = (v_2,0)$ for arbitrary $v_2 \in \ker \dif \s$. Then
\[ 0 = \Omega^\P((v_1,w_1),(v_2,0)) = \Omega^\G(v_1,v_2) \, . \]
Therefore, we can conclude that $v_1$ is $\Omega^\G$ orthogonal to $\ker \dif \s$. Since $\G$ is a symplectic groupoid, this implies that $v_1 \in \ker \dif \t$.

Now suppose $(v_2,w_2) = (\dif \u \dif f(w_2),w_2)$ for arbitrary $w_2 \in T \O_x$ tangent to orbit of $x$.
We can conclude that
\begin{equation}\label{eqn:symform1}
 0 = \Omega^\G(v_1, \dif \u \dif f(w_2)) + \beta(w_1,w_2) \, .
\end{equation}
Let $\G_{\O_y}=\s^{-1}(\O_y)=\t^{-1}(\O_y)$ be the restriction of $\G$ to the orbit $\O_y$ and let $\omega^{\O_y}$ be the leafwise symplectic form on the orbit.
For any symplectic groupoid, it turns out that
\[ \Omega^\G|_{\G_{\O_y}} = \t^* \omega^{\O_y} - \s^* \omega^{\O_y}. \]
Since both $v_1$ and $\dif \u \dif f(w_2)$ are tangent to $\G_{\O_y}$, we can conclude that
\begin{align*}
\Omega^\G(v_1, \dif \u \dif f(w_2)) &= \omega^{\O_y}(\dif \t (v_1), \dif f(w_2)) - \omega^{\O_y}(\dif \s (v_1), \dif f(w_2)) \\
                                   &= -\omega^{\O_y}(\dif \s (v_1), \dif f(w_2)) \\
                                   &= -\omega^{\O_y}(\dif f (w_1), \dif f(w_2)) \\
                                   &= -f^*\omega^{\O_y}(w_1,w_2) \\
                                   &= -\omega^{\O_x}(w_1,w_2) - \beta(w_1,w_2).
\end{align*}
In the second line we have use the fact that $v_1 \in \ker \dif \t$.
In the third line we used the fact that $(v_1,w_1)$ must be tangent to $\G \times_M U$.
The last line follows from the fact that $(f,\beta)$ is a morphism of Dirac manifolds.

Combining this with (\ref{eqn:symform1}) we get that
\begin{equation}\label{eqn:symform2}
 \omega^{\O_x}(w_1,w_2) = 0 \, .
\end{equation}
Recall that $w_2$ was an arbitrary vector tangent to $\O_x$.
Since $\omega^{\O_x}$ is symplectic, we can conclude that $w_1=0$.

So far we have shown that $(v_1,w_1) = (v_1,0)$ and that $v_1 \in \ker \dif \t$. It follows that $v_1 \in \ker \dif \s$.
We still need to show that $v_1= 0$.
To do this we will show that $\Omega^\G(v_1,v) = 0$ for arbitrary $v \in T_{\u(y)} \G$.
Since $\Omega^\G$ is symplectic, this will show that $v_1=0$.

First write $v$ in the form $v_A + v_\u$ for $v_A \in \ker \dif \s$ and $v_\u \in \mbox{Im}(\dif \u)$.
Let us see what happens when we pair it with $v_1$:
\[ \Omega^\G(v_1,v) = \Omega^\G(v_1,v_A) + \Omega^\G(v_1, v_\u)  = 0 + \Omega^\G(v_1, v_\u) \, . \]

Since $f$ is transverse to the foliation on $M$, we can write $v_\u = \dif \u \dif f(w) + \dif \u (v_\O)$, where $w \in T_x N$ and $v_\O \in T_y \O_y$.
Hence,
\begin{equation}\label{eqn:symform3}
\Omega^\G(v_1, v_\u) = \Omega^\G(v_1, \dif u \dif f(w)) + \Omega^\G(v_1, \dif \u(v_\O)) \, .
\end{equation}
Recall that we have assumed that $(v_1,0)$ is in the kernel of $\Omega^P$.
Therefore,
\begin{align*}
0 &= \Omega^P((v_1,0),(\dif{\u{}} \dif f (w),w)) \\
&= \Omega^\G(v_1, \dif \u \dif f(w)) + \beta(0,w) \\
 &= \Omega^\G(v_1, \dif \u \dif f(w)) \, .
\end{align*}
Therefore, we can conclude that the first summand on the right side of (\ref{eqn:symform3}) vanishes.
For the second summand, observe that both $v_1$ and $\dif\u(v_\O)$ are tangent to $\G_{\O_y}$ and so
\[ \Omega^\G(v_1,\dif\u(v_\O)) = \omega^{\O_x}(\dif \t (v_1), v_\O) -\omega^{\O_x}(\dif \s (v_1), v_\O) = 0 \, . \]
Hence, we conclude that $\Omega^\G(v_1,v)=0$.
Since $v$ was arbitrary and $\Omega^\G$ is symplectic, we conclude that $v_1 =0$.
So $\Omega^P$ is non-degenerate at $p$ and therefore symplectic.
\qed\end{proof}
%
%
\subsection{Weak equivalences}\label{subsection:weakequivalences}
Let $\G$ and $\H$ be D-Lie groupoids over $M$ and $N$.
Suppose $F: \H \to \G$ is a morphism of D-Lie groupoids covering $f: N \to M$.
Then we can construct the left principal $(\G, \H)$-bibundle:
\[ P_F := \G \times_{\s,f} N \, , \]
with the obvious commuting actions of $\G$ (on the left) and of $\H$ (on the right). We equip $P_F$ with the characteristic form:
\[ \Omega^{P_F} = \pr_1^* \Omega^\G + \pr_2^* \beta \, , \]
where $\beta$ is the gauge part of $f: N \to M$.
This is the same as the standard construction for Lie groupoids with the addition of the characteristic form.
The reader can easily check that these actions satisfy the axioms of a $(\G, \H)$-bibundle.

\begin{definition}
We say that a morphism of D-Lie groupoids $F: \H \to \G$ is a \emph{weak equivalence} if $P_F$ is a principal $(\G,\H)$-bibundle.
\end{definition}

In other words, a weak equivalence of D-Lie groupoids is a D-Lie groupoid morphism which gives rise to a Morita equivalence.
Later, it will be shown that these equivalences further correspond to an isomorphism of the underlying stacks.
The name weak equivalence is chosen because while they are not necessarily invertible as D-Lie groupoid morphisms, they become (weakly) invertible when passing to the 2-category of stacks.

Recall that $P_F$ is principal if and only if it is a principal bibundle of Lie groupoids.
Therefore, $F$ is a weak equivalence if and only if the underlying generalized map of Lie groupoids is a weak equivalence.
This immediately gives rise to a notion of symplectic weak equivalences.
\begin{example}[Symplectic Weak Equivalences]
Suppose $\G$ and $\H$ are symplectic groupoids (i.e., $\G$ and $\H$ are target aligned D-Lie groupoids and their characteristic 2-forms $\Omega^\G$ and $\Omega^\H$ are symplectic).
Then a weak equivalence $F: \H \to \G$ consists of a homomorphism of Lie groupoids, together with a closed 2-form $\beta$ on $N$ such that the following hold.
\begin{enumerate}[(a)]
\item $F: \H \to \G$ is fully faithful and essentially surjective.
\item $f: N \to M$ is transverse to $\pi_M$ (the Poisson structure on $M$).
\item $F^* \Omega^\G = \Omega^\H + \t^* \beta - \s^* \beta$.
\end{enumerate}
Condition (a) and (b) ensure that $F$ is a weak equivalence of the underlying Lie groupoids as per the usual definition.
That is, $P_F$ is principal as a Lie groupoid bibundle.
The last condition is the geometric condition for $F$ to consitute a morphism of D-Lie groupoid as per our discussion of D-Lie groupoid morphisms.
\end{example}
Composition of homomorphisms corresponds to the tensor product operation at the level of bimodules.
Given a left principal $(\G_1, \G_2)$-bibundle $P$ and a left principal $(\G_2, \G_1)$-bibundle $Q$.
Assume that $\G_1$, $\G_2$, $\G_3$, $P$ and $Q$ are all target aligned.
Thinking of the $\G_i$ as Lie groupoids then
\[ P \otimes Q := P \times_{M_2} Q / \G_2 \, , \]
where the action of $\G_2$ on $(p,q)$ is defined to be $g_2 \cdot (p,q) = (p \cdot g_2\inv , g \cdot q)$.
In order to equip $P \otimes Q$ into a target aligned left principal $(\G_1,\G_3)$-bibundle, we only need to equip it with a multiplicative 2-form $\Omega^{P \otimes Q}$.
Multiplicativity of $\Omega^P$ and $\Omega^Q$ with respect to the action of $\G_2$ ensures that
\[ \widetilde{\Omega} := \pr_1^* \Omega^P + \pr_2^* \Omega^Q \, , \]
is basic with respect to the action of $\G_2$ on $P \times_{M_2} Q$.
Hence, $\widetilde{\Omega}$ descends to a 2-form on $P \otimes Q$.
Left and right multiplicativity of $\Omega^{P \otimes Q}$ can easily be checked.

The following standard facts for Lie groupoids also hold for D-Lie groupoids.
\begin{itemize}
\item There is a 2-category whose objects are D-Lie groupoids, 1-morph\-isms are left principal bibundles, and 2-morphisms are bibundle isomorphisms.
\item The mapping $F \mapsto P_F$ is functorial (i.e. $P_{F \circ G} \isom P_F \otimes P_G$).
\item $P_F$ has a (weak) inverse if and only if $F$ is a weak equivalence.
\end{itemize}
\section{Stacks over \texorpdfstring{$\DMan$}{}}\label{section:stacks}
After all the previous apparatus, the treatment of stacks over $\DMan$ is nearly identical to the case of $\Man$.
We will borrow the notation and some proof outlines from~\cite{PXstacks} where the case of stacks over $\Man$ is treated.
To that end, our definition of a Grothendiek (pre)-topology, categories fibered in groupoids, and stacks will be essentially identical to those in~\cite{PXstacks} where we replace the site of manifolds with the new site of Dirac manifolds.

\subsection{Stacks and CFGs}

We will review a few basic definitions which will be useful for us in the sequel.
A \emph{category fibered in groupoids} (CFG) over a category $\C$ is a category $\X$ together with a functor $\pi: \X \to \C$ such that the following properties hold.
\begin{enumerate}[(C1)]
\item Given any morphism $f: M \to N$ in $\C$ and object $x$ in $\X$ such that $\pi(x) = N$, then there exists a object $y$ in $\C$ and a morphism $a: y \to x$ in $\X$ such that $\pi(a)=f$.
\item Given morphisms $f: M_1 \to M_2$ and $g: M_2 \to M_3$ in $\C$ together with $a:x \to z$ and $b:y \to z$ such that $\pi(a) = g \circ f$ and $\pi(b) = g$, then there exists a unique morphism $c$ such that $\pi(c) = f$ and $b \circ c = a$.
\end{enumerate}

A morphism of CFGs $\F: \X \to \Y$ is a functor which commutes with the projections to $\C$.
A morphism of CFGs is called an \emph{isomorphism} if it is an equivalence of categories.
A 2-morphism $\eta: \F_1 \to \F_2$ is a (necessarily invertible) natural transformation of functors.
Formally, this notion of morphisms and 2-morphisms makes CFGs over $\C$ into a strict 2-category with invertible 2-morphisms.
We will think of CFGs in this manner, so if we write $\X_1 \times_\Y \X_2$ a fiber product of CFGs, then we mean the 2-categorical fiber product (see~\cite{Mstacks} for an explicit construction of this operation).

We can associate to any object of $\C$ a CFG by letting $\X := \Hom_{\C}(-,M)$ (this is also known as the slice category of $M$).
This gives a fully faithful embedding of $\C$ into the 2-category of CFGs over $\C$.
A CFG is called \emph{representable} if it is isomorphic to some object in $\C$.

The object whose existence is asserted by (C1) is unique up to a unique isomorphism and is frequently called the \emph{pullback} of $x$ along $f$.
It is often convenient to make a choice of pullback which we usually denote $f^*x$ or $x|_{\pi(y)}$.
Making such a choice, allows us to identify objects of $\X$ over $M$ with morphisms of CFGs $M \to \X$ (see the remark at the end of section 2.1 in~\cite{PXstacks}).

Given $x: M \to \X$ and $y: N \to \X$ we call the associated fiber product $M \times_\X N$ (denoted $\Isom(x,y)$ in~\cite{PXstacks}) the \emph{symmetry bibundle} of $x$ and $y$.
When $x = y$ then $M \times_\X M$ is the \emph{symmetry groupoid} of $x$ and is canonically the space of arrows of a (strict) groupoid internal to CFGs\footnote{A more explicit construction of the groupoid structure can found in the proof of \emph{70. Proposition} in~\cite{Mstacks}.
The relevant portion can be found in the paragraph beginning with `Now we show the converse'.}.
The source and target maps of this groupoid structure are the right and left projections, respectively.
The multiplication morphism is constructed by observing that there is a canonical isomorphism
\[ (M \times_X M) \times_M (M \times_X M) \to M \times_\X M \times_\X M \, , \]
and then composing it with
\[ \pr_1 \times \pr_3: M \times_\X M \times_\X M \to M \times_\X M \, . \]

With the basics and notation for CFGs out of the way, we can now introduce stacks.
For this, we will need to equip $\C$ with a Grothendiek pre-topology.
A Grothendiek pre-topology is an assignment to each object $M$ in $\C$ of a collection of subsets of the set $\Hom(-,M)$ called \emph{covering families}.
This assignment must satisfy some properties.
\begin{enumerate}[(T1)]
\item If $f: M \to N$ is an isomorphism then $\{ f \}$ is a covering family.
\item If $\{ u_i: U_i \to M \}$ is a covering family of $M$ and $g: N \to M$ is any morphism, then $\{ \pr_2: U_i \times_M N \to N \}$ is a covering family of $N$.
\item If ${\{ f_i: U_i \to M \}}_{i \in I}$ is a covering families of $M$ and ${\{ V_{ij} \to U_i \}}_{j \in J_i}$ is covering family of $U_i$ for each $i$, then the compositions $\{ V_{ij} \to U_i \to M \}$ constitute a covering family of $M$.
\end{enumerate}

If $\C = \Man$, then we can give $\Man$ the following pre-topology: a covering $\{ u_i: U_i \to M \}$ of manifold $M$ is a collection of \'etale smooth maps $u_i$ whose images cover $M$.
This is the same pre-topology used in~\cite{PXstacks}.
Using the forgetful functor $\Pr_1: \DMan \to \Man$, we can also define a Grothendiek pre-topology on $\DMan$.
A collection of morphisms $\{ u_i: U_i \to M \}$ in $\DMan$ is a \emph{covering} of $M$ if and only if $\{ \Pr_1(u_i) \}$ is a covering of $\Pr_1(M)$.
Using our construction of the fiber product in $\DMan$, it is straightforward to check that this is a well defined Grothendiek pre-topology.

\begin{definition}
A CFG $\X$ over $\C$ is called a \emph{stack} if it satisfies:
\begin{enumerate}[(S1)]
\item Suppose $\{ i_a: U_a \to M \}$ is a covering of $M$.
Let $\{ P_a \}$ be a collection of objects in $\X$ over $U_a$ together with morphisms $\phi_{ab}: P_b|_{U_{ab}} \to P_a|_{U_{ab}}$ such that $\phi_{ab} \circ \phi_{bc} = \phi_{ac}$ when restricted to any triple intersection $U_{abc}$.
Then there exists an object $P$ over $M$ together with morphisms ${\{ \phi_a: P|_{U_a} \to P_a \}}_{a \in A}$ such that $\phi_{ab} \circ \phi_{b} = \phi_a$.
\item Let $P$ and $Q$ be objects in $\X$ over $M$ and suppose $\{ i_a : U_a \to M \}$ is a covering of $M$ in $\C$.
Suppose we have a collection of morphisms $F_a: P|_{U_a} \to Q|_{U_a}$ covering the identity such that:
\[ F_a|_{P|_{U_{ab}}} = F_b|_{Q|_{U_{ab}}} \, , \]
then there exists a unique morphism $F: P \to Q$ such that $F|_{P|_{U_a}} = F_a$.
\end{enumerate}
\end{definition}
From now on, we will always assume that all CFGs or stacks are over $\C = \DMan$ with the pre-topology described above.

We conclude this section with the following important proposition.
\begin{proposition}
Suppose $\G$ is a D-Lie groupoid.
Let $\B\G$ denote the category whose objects are left principal $\G$-bundles and morphisms are equivariant maps $F: P \to Q$.
Let the functor $\pi: \B \G \to \DMan$ send $F: P \to Q$ to $f:M \to N$.
Then $\B\G$ is a stack.
\end{proposition}
\begin{proof}
That $\B\G$ satisfies the axioms of a CFG is the content of Lemma~\ref{lemma:gbundlepullback} and Lemma~\ref{lemma:gbundlecartesian}.
Furthermore, $\B \G$ is satisfies the axioms of a stack by Lemma~\ref{lemma:gbundleglueing}.
\qed\end{proof}
\subsection{Presentations and groupoids}
To define a presentation of a stack, we need to define a special class of morphisms which play the role of `surjective submersions' of stacks.

\begin{definition}
Suppose $F: \X \to \Y$ is a morphism of stacks.
We call $F$ a \emph{representable epimorphism} if and only if
\begin{enumerate}[(a)]
\item \emph{representable:} given any stack morphism $N \to \Y$ where $N \in \DMan$ is representable, then $\X \times_\Y N$ is representable and
\item \emph{epimorphism:} given any $y \in \Y$ over $M \in \DMan$, there is a covering $\{ U_i \}$ of $M$ such that $y|_{U_i}$ is in the image of $F$.
\end{enumerate}
\end{definition}

It is a standard fact that that representable epimorphisms of CFGs are stable under base changes.
Furthermore, one can check without much difficulty that a representable epimorphism of representable stacks is a surjective submersion.

Now suppose $\X$ is a stack and $M$ is an object of $\DMan$.
A representable epimorphism $p:M\to \X$ is called a \emph{presentation} of $\X$.
In such a case we may say that the stack $\X$ is \emph{Dirac differentiable} or \emph{presentable}.
When $p: M \to \X$ is the morphism associated to some object $x \in \X$ then we call $x$ a \emph{versal family} of $\X$.
In general, a stack admits a versal family if and only if it is presentable.
\begin{lemma}
Suppose $\G$ is a D-Lie groupoid.
Then $\B\G$ admits a versal family.
\end{lemma}
\begin{proof} The proof is similar to the usual case of $\Man$:

Consider $\G$ as a trivial $\G$-bundle over its space of objects $M$. To any morphism $f: N \to M$ we can pullback $\G$ to obtain the `trivial' $\G$ bundle associated to the map $f$.
This gives us a functor $M \to \B\G$ as CFGs.
Now consider any other CFG map $N \to \B\G$ where $N$ is representable.
Without loss of generality, we can assume that the map $N \to \B\G$ is the morphism corresponding to some $P \in \B\G$ over $N$.
Furthermore, we can assume that $N$ is small and therefore that $P = g^* \G$ is a trivial principal bundle associated to some map $g: N \to M$. By constructing $M \times_{\B\G} N $ explicitly, we see that it is isomorphic to $\Hom(-,\G \times_{\s,g} N)$. Therefore $M \to \B\G$ is representable.
It is also clearly an epimorphism since every $\G$ bundle is locally trivial.
\qed\end{proof}
As with Lie groupoids, this example is actually the universal case.
\begin{proposition}
Suppose $\X$ is a Dirac differentiable stack and $x: M \to \X$ is a versal family of $\X$, then $\G:=M \times_X M$ is a Lie groupoid and $\X \isom \B\G$.
\end{proposition}
\begin{proof}
For the first part, note that by definition $M \times_X M$ must be representable.
Furthermore, we observed earlier that it is a groupoid internal to stacks over $M$.
Finally since the projections $\pr_i: M \times_X M \to M$ are surjective submersions, we can conclude that $\G$ is a D-Lie groupoid.

For the second part, the proof can proceed identically to the proof of Theorem 2.22 in~\cite{PXstacks} where we replace the site of smooth manifolds with $\DMan$.
Note that this proof works by constructing categorical objects, and then observing that since they are representable stacks, they must coincide with geometric objects.
For example, when the underlying site is smooth manifolds a representable principal $\G$-bundle internal to the category of stacks is just an ordinary principal bundle.
Since we have defined D-Lie groupoids and principal bundles of D-Lie groupoids in a purely categorical way, the existing proof can be used without modification.
\qed\end{proof}
The next theorem shows that the notion of Morita equivalence and stack isomorphism are equivalent.
\begin{theorem}\label{thm:stackequiv}
Suppose $\G \rightrightarrows M$ and $\H \rightrightarrows N$ are D-Lie groupoids.
Then the following are equivalent:
\begin{enumerate}[(i)]
\item $\B \G$ and $\B \H$ are isomorphic.
\item There exists a principal $(\G, \H)$-bibundle.
\item There exists a D-Lie groupoid $\K$ and weak equivalences $F_1: \K \to \G$ and $F_2:\K \to \H$.
\end{enumerate}
\end{theorem}
\begin{proof}
This result can be thought of as an analogue of Theorem 2.26 in~\cite{PXstacks}.

$(i) \Rightarrow (ii)$.
Let $\F: \B \G \to \B\H$ be an isomorphism.
Now consider the representable stack $P:= M \times_{\B\H} N \isom M \times_{\B\G} N$.
$P$ inherits a canonical principal left and right action of $\G$ and $\H$ respectively.
Hence $P$ is a principal $(\G,\H)$-bibundle.

$(ii) \Rightarrow (iii)$.
Now suppose we are given a principal $(\G,\H)$-bibundle $P$.
Let $\K$ be defined to be the groupoid constructed by pulling back $\H$ along $\s^{P}$
\[ \K := P \times_{\s^P,\t^\H} \H \times_{\s^\H, \s^P} P  \, .\]
Let $F_1: \K \to \H$ be defined to be projection to the center component.
The characteristic form of $\K$ can be constructed by pulling back the characteristic form of $\H$ along the central projection.
It is easy to check that this is still a multiplicative form and so $\K$ is a D-Lie groupoid.
Furthermore, we can see that $F_1$ is a weak equivalence.

Now we define $F_2: \K \to \G$ at the level of sets to be the unique map satisfying $F_2(p,h,q) \cdot p = q \cdot h$.
That this is a well defined weak equivalence follows from the fact that the left and right multiplication maps on $P$ are principal.
Finally, using the definition of $F_2$ we observe that $F_2$ satisfies
\[ F_2^* \Omega^\G = \Omega^{\K} + {(\t^{\K})}^*\Omega^P - {(\s^{\K{}})}^* \Omega^P \]
and so $F_2$ can be viewed as a morphism of D-Lie groupoids.

$(iii) \Rightarrow (i)$.
It suffices to show that given any weak equivalence of D-Lie groupoids $F: \G \to \H$, we can construct a stack isomorphism $\til F: \B \G \to \B\H$.
Recall that given any morphism, there is an associated invertible bimodule $(\H,\G)$-bibundle $P_F$.
Then we define $\til F(P):= P_F \otimes P$.
Since $P_F$ is invertible, this functor has a weak inverse of the form $P \mapsto (P_F)\inv \otimes P$ and hence is an isomorphism of stacks.
\qed\end{proof}
\begin{example}[The Stack of a Poisson Manifold]
We saw earlier that there is a one-to-one correspondence between symplectic groupoids $(\G, \Omega)$ integrating a Poisson manifold $(M, \pi)$ and target aligned D-Lie groupoid $\G$ with a symplectic characteristic form.
Furthermore, notice that Theorem~\ref{thm:main1} is an immediate corollary of Proposition~\ref{prop:symform} and Theorem~\ref{thm:stackequiv}.

From this point of view, if $\G$ is a proper symplectic groupoid then $\B\G$ is the analogous in $\DMan$ of the notion of a \emph{separated stack}.
The space of objects of such proper symplectic groupoids are the \emph{Poisson manifolds of compact type}, studied by Crainic, Fernandes and Martinez-Torrez in~\cite{PMCT1,PMCT2}.
\end{example}
\begin{example}[A non-presentable stack]
One weakness of the category $\DMan$ is that it lacks a terminal object.
This is remedied by passing to stacks over $\DMan$.
In fact, the category $\DMan$ equipped with the identity projection is a terminal object in the 2-category of stacks.
This is, perhaps, the simplest example of a stack over $\DMan$ which does not admit a presentation.
\end{example}
The above results should convice the reader that our notion of stack suitably captures the existing theory of symplectic Morita equivalences.
In the next section, we will see that these results also lead to a natural notion of \emph{infinitesimal} symplectic Morita equivalence.
\section{Infinitesimal weak equivalences}\label{section:infweakequivalences}
To treat infinitesimal weak equivalences of Poisson manifolds, we first make a few comments about general algebroids.
We will say that a morphism of algebroids $F: A \to B$ is \emph{transverse} if it covers a smooth map $f: M \to N$ which is transverse to the orbit foliation of $B$.
Our goal is to provide an infinitesimal criteria for a morphism of Lie algebroids to integrate to a weak equivalence of Lie groupoids.
We begin with a standard lemma.
\begin{lemma}\label{lemma:justisolemma}
Let $A$ and $B$ be integrable algebroids and suppose $F: A \to B$ is an algebroid morphism covering $f: M \to N$.
Then $F$ integrates to a weak equivalence $\F: \Sigma(A) \to \Sigma(B)$ if and only if the following hold for all $x \in M$:
\begin{enumerate}[(a)]
\item $F$ is transverse;
\item $F$ induces a homeomorphism of orbit spaces $M / A \to N / B$;
\item $\F_x: {\Sigma(A)}_x \to {\Sigma(B)}_x$ is an isomorphism.
\end{enumerate}
\end{lemma}
\begin{proof}
Recall that a groupoid morphism $\F:\Sigma(A) \to \Sigma(B)$ is a weak equivalence if
\[ P_\F := \Sigma(B) \times_{\s,f} M \]
is a principal bibundle.
The left action is always principal, so the only requirement is that the right action is also principal.
That is, $\t^{P_\F}:P_\F\to M$ is a surjective submersion such that the right action is free and transitive over its fibers.

First, we observe that:
\begin{itemize}
  \item $\t^{P_\F}:=\t \circ \pr_1: \Sigma(B) \times_{\s,f} M \to M$ is a surjective submersion if and only if $f$ is transverse to the foliation of $M$ and $\F$ is surjective at the level of orbits.
  This is a fairly straightforward fact to check.
  Surjectivity comes from the orbit condition while the transversality condition ensures that the map is a submersion.
  \item The right action of $\Sigma(A)$ is free if and only if $\F_x$ is injective.
  This follows immediately from the definition of the right action.
\end{itemize}
We claim that the right action of $\Sigma(A)$ is transitive over the fibers of $\t^{P_\F}$ if and only if $\F$ is injective at the level of orbits and $\F_x$ is surjective, which will complete the proof:

$(\Leftarrow)$. Suppose $(g_1,x_1)$ and $(g_2,x_2)$ are in the same $\t^{P_\F}$ fiber.
Then $g_1\inv g_2$ is an arrow from $f(x_2)$ to $f(x_1)$.
Since $\F$ is injective at the level of orbits, $x_1$ and $x_2$ must be in the same $\Sigma(A)$ orbit.
Furthermore, since $\F_x$ is surjective, then there must exist some $h: x_1 \to x_2$ such that $f(h) = g_1\inv g_2$.
The definition of the right action of $\Sigma(A)$ makes it clear that $(g_1,x_1) \cdot h = (g_2,x_2)$ and so the action is transitive.

$(\Rightarrow)$. Suppose the right action is transitive.
If $f(x_1)$ and $f(x_2)$ lie in the same $\Sigma(B)$ orbit then there must be some $g: f(x_2) \to f(x_1)$.
Since the action is transitive, there must be some $h:x_2 \to x_1$ such that $(1_{f(x_1)},x_1) \cdot h = (g,x_2)$ and so $x_1$ and $x_2$ are in the same orbit.
This shows the map of orbit spaces is injective.
$\F_x$ must be surjective since for any $g \in {\Sigma(B)}_{f(x)}$ there must exist an $h$ such that $(1_{f(x)}, x) \cdot h = (g,x)$ which implies that $h$ is a preimage of $g$.
These three bullets together show our lemma if we replace homeomorphism with continuous bijection.
However, transversality of $\F$ implies that the map of orbit spaces is also open in a manner analgous to the fact that submersions are open maps.
\qed\end{proof}

This proposition shows that, in order to obtain an infinitesimal criteria for $\F$ to be a weak equivalence, we need to understand under what conditions $\F_x$ is an isomorphism.
So we turn to this question.
\subsection{Monodromy}\label{subsection:monodromy}
We need to recall the \emph{monodromy groups} of Crainic and Fernandes~\cite{Cint}. We will use the same notations as in~\cite{Cint}.
So, given a Lie algebroid $A$, we denote by $\G(\g_x)$ the source simply connected integration of the isotropy Lie algebra of $A$ at $x \in M$.
\begin{definition}\label{defn:monodromy}
Let $A$ be an integrable Lie algebroid.
The \emph{monodromy} of $A$ at $x \in M$ is the kernel $\N_x(A)$ of the canonical map $\G(\g_x) \to {\Sigma(A)}_x$.
\end{definition}
If we think of elements of $\G(\g_x)$ as $\g_x$-paths modulo $\g_x$-homotopy, the map $\G(\g_x) \to {\Sigma(A)}_x$ is the passage to $A$-homotopy.
The group $\N_x(A)$ fits into a short exact sequence.
\[
\begin{tikzcd}
1 \arrow[r] & \N_x(A) \arrow[r] & \G(\g_x) \arrow[r] &  {\Sigma(A)}^\circ_x \arrow[r] & 1.
\end{tikzcd}
\]

A morphism of algebroids $F: A \to B$ induces a map $\til F: \G(\g_x) \to \G(\g_{f(x)})$ and for any $g \in \G(\g_x) \in \N_x(A)$ we always have that $\til F(g) \in \N_{f(x)}(B)$.
Therefore, if $F: A \to B$ is a morphism of Lie algebroids, then for each $x \in M$ we obtain a commutative diagram.
\[
\begin{tikzcd}
1 \arrow[r] & \N_x(A) \arrow[r] \arrow[d]  & \G(\g_x) \arrow[r] \arrow[d, "\til F"] &  {\Sigma(A)}^\circ_x \arrow[r] \arrow[d, "\F_x"]  & 1 \\
1 \arrow[r] & \N_{f(x)}(B) \arrow[r]       & \G(\g_{f(x)}) \arrow[r]                &  {\Sigma(B)}^\circ_{f(x)} \arrow[r]             & 1
\end{tikzcd}
\]

This leads immediately to:
\begin{lemma}\label{proposition:connectedcomponentofidentity}
Suppose $F: A \to B$ is an algebroid morphism of integrable algebroids and let $\F: \Sigma(A) \to \Sigma(B)$ be the integration of $F$.
Then the restriction $\F_x:{\Sigma(A)}^\circ_x \to {\Sigma(B)}^\circ_{f(x)}$ is an isomorphism if and only if both $\til F: \N_x(A) \to \N_{f(x)}(B)$ is an isomorphism and $F_x: \g_x \to \g_{f(x)}$ is an isomorphism.
\end{lemma}
To obtain a condition for an isomorphism of the full isotropy groups, we observe that the group of connected components of ${\Sigma(A)}_x$ can be identified with $\pi_1(\O_x)$. In particular, for any algebroid $A$ we have another short exact sequence.
\[
\begin{tikzcd}
1 \arrow[r] & {\Sigma(A)}^\circ_x \arrow[r] & {\Sigma(A)}_x \arrow[r] & \pi_1(\O_x) \arrow[r] & 1
\end{tikzcd}
\]
Again, if $F: A \to B$ is a morphism of Lie algebroids, we get maps,
\[
\begin{tikzcd}
1 \arrow[r] & {\Sigma(A)}^\circ_x \arrow[r] \arrow[d, "\F"] & {\Sigma(A)}_x \arrow[r] \arrow[d, "\F"] & \pi_1(\O_x) \arrow[d, "f_*"] \arrow[r] & 1 \\
1 \arrow[r] & {\Sigma(B)}^\circ_{f(x)} \arrow[r] & {\Sigma(B)}_{f(x)} \arrow[r] & \pi_1(\O_{f(x)}) \arrow[r] & 1
\end{tikzcd}
\]
This leads immediately to:
\begin{lemma}
Suppose $F: A \to B$ is a transverse morphism of algebroids.
Suppose further that $\F_x: {\Sigma(A)}_x \to {\Sigma(B)}_{f(x)}$ restricts to an isomorphism of the connected components of the identity.
Then $\F_x$ is an isomorphism if and only if ${(f|_{\O_x})}_*: \pi_1(\O_x) \to \pi_1(\O_{f(x)})$ is an isomorphism.
\end{lemma}
We can now give a short proof of Theorem~\ref{thm:infweakequiv}.
\begin{proof}[of Theorem~\ref{thm:infweakequiv}]
The two preceeding propositions togeather imply that (c-e) are satisfied if and only if $F$ integrates to an isomorphism at the level of isotropy groups.
If we combine these facts with Lemma~\ref{lemma:justisolemma} we immediately arrive at the result.
\qed\end{proof}

We call an algebroid morphism $F: A \to B$ satisfying (a-e) a \emph{weak equivalence}.
Now observe that the monodromy of $A$ at $x$ is also defined for non-integrable algebroids.
In fact, the failure of $\N_x(A)$ to be discrete measures the failure of $\Sigma(A)$ to be smooth~\cite{Cint}.
Therefore, the definition of weak equivalence makes sense even when $A$ and $B$ are not integrable.
It is not yet known whether this definition of weak equivalence is fully satisfactory.
It would be hoped that such maps correspond to equivalences of higher categorical objects, and this is certainly true in the integrable case.
However, it is not completely clear what sort of stack or higher categorical object one should associate to a non-integrable algebroid and it is beyond the scope of this paper to discuss this issue (however, see~\cite{BStacky}).
\subsection{Weak equivalences of D-Lie algebroids}
Theorem~\ref{thm:infweakequiv} has the following immediate consequence:
\begin{corollary}
Suppose $F: A \to B$ is a morphism of integrable D-Lie algebroids.
Then $F$ integrates to a weak equivalence if and only if $F$ satisfies the (a-e) of Theorem~\ref{thm:infweakequiv}.
\end{corollary}
We can use this to define weak equivalences for D-Lie algebroids.
\begin{definition}
  A morphism of D-Lie algebroids $F: A \to B$ is a \emph{weak equivalence} when the underlying algebroid morphism is a weak equivalence.
  That is, $F$ satisfies (a-e) of Theorem~\ref{thm:infweakequiv}.
\end{definition}
\begin{example}[Poisson Manifolds]
Suppose $\pi_M$ and $\pi_N$ are Poisson structures on $M$ and $N$.
Then the corresponding Dirac structures $L_M$ and $L_N$ are also D-Lie algebroids.
Therefore, we call $f: M \to N$ a weak equivalence of Poisson manifolds if $f$ is a weak equivalence of their corresponding D-Lie algebroids.
\end{example}

We can now prove Theorem~\ref{thm:main2}.
\begin{proof}[of Theorem~\ref{thm:main2}]
First observe that, if $F: L_X \to L_M$ is a weak equivalence, we have
\[ L_X \isom f^! L_M := L_M \times_{\rho, \dif f} X \, . \]
This construction is sometimes called the \emph{pullback algebroid}.
When $L_M$ is integrable, then the pullback $f^! L_M$ is integrated by
\[ f^! \Sigma(M) := X \times_{f,\t} \Sigma(M) \times_{\s,f} X \, . \]
Therefore, if $f: X \to M$ is a weak equivalence and $M$ is integrable, then $X$ is integrable.
Hence, if $M \from X \to N$ are a pair of weak equivalences, then $M$ and $N$ are certainly Morita equivalent.

On the other hand, suppose $M$ and $N$ are Morita equivalent.
Then there exists a symplectic bibundle $M \from P \to N$.
Let $X := P$.
Since the fibers of $\t^P:P \to M$ and $\s^P :P \to N$ are simply connected and $\s^P$ and $\t^P$ submersions:
\[ {(\s^P)}^! \Sigma(M) \isom {(\t^P)}^! \Sigma(N) \isom \Sigma(X) \, . \]
Therefore, $\s^P$ and $\t^P$ are the unit maps of weak equivalences of D-Lie groupoids.
So they must be weak equivalences of Dirac manifolds.
\qed\end{proof}
\appendix
\section{Proof of Lemma \texorpdfstring{\ref{lemma:dliedata}}{}}\label{appendix1}
\begin{proof}
Suppose $\G$ with $\s,\t,\u,\m,\i$ is a Lie groupoid and $(\s,\sigma),(\t,\tau),(\u,\upsilon),(\i,\iota)$ are morphisms in $\DMan$.
This data constitutes a D-Lie groupoid if and only if the gauge equation associated to each groupoid axiom holds.
In the table below, we have enumerated the axioms of a groupoid and computed the corresponding equations of 2-forms.
\begin{center}
\rowcolors{1}{gray!25}{white}
\begin{tabular}{ |c|c|c|c|}
\hline\noalign{\smallskip}

            & Axiom & Domain & Gauge Part  \\
\noalign{\smallskip}\hline\noalign{\smallskip}
(G1) & $\s \circ \u = \Id_M$ & $M$ & $\u^* \sigma + \upsilon = 0$\\
(G2) & $\s \circ \m = \s \circ \pr_2$ &$ \G^{(2)}$& $\m^* \sigma + \mu = \pr_1^* \sigma + \pr_2^* \sigma$\\
(G3) & $\s \circ \i = \t$&$ \G$& $\i^* \sigma + \iota = \tau$\\
(G4) & $\i \circ \u = \u$&$ M$ &$\u^* \iota + \upsilon = \upsilon$ \\
(G5) & $\m \circ ((\u \circ \t) \times \Id_\G) = \Id_\G$&$ \G$ & ${((\u \circ \t) \times \Id)}^* \mu = {(\u \circ \t)}^* \sigma$  \\
(G6) & $\m \circ (\Id_\G \times (\u \circ \s)) = \Id_\G$&$ \G$& ${(\Id \times (\u \circ \s) )}^* \mu = {(\u \circ \s)}^* \tau$\\
(G7) & $\m \circ (\i \times \Id_\G) = \u \circ \s$&$ \G$&${(\i \times \Id)}^* \mu - \i^* \sigma = \s^* \upsilon + \sigma$ \\
(G8) & $\m \circ (\Id_\G \times \i) = \u \circ \t$&$ \G$&${(\Id \times \i)}^* \mu - \i^* \tau = \t^* \upsilon + \tau$\\
(G9) & ${\!\begin{aligned} & \m\circ (\m (\pr_1 \times \pr_2) \times \pr_3) = \\
                           & \m \circ (\pr_1 \times \m(\pr_2 \times \pr_3))
                           \end{aligned}}$ & $\G^{(3)}$ &  see (\ref{eqn:assocgauge}) below.\\
                           \noalign{\smallskip}\hline
\end{tabular}
\end{center}
Now suppose we are supplied with 2-forms $\sigma$ and $\tau$ satisfying (i) and (ii) from~\ref{lemma:dliedata}.
Take the gauge equations from (G1-G3) to be the definitions of $\upsilon$, $\mu$ and $\iota$.
Let $L_\G := \s^* L_M - \sigma$.
We must show that $\G$ and $M$ together with $(\s,\sigma),(\t,\tau),(\u,\upsilon),(\i,\iota)$ constitutes a well defined D-Lie groupoid.
Assumption (i) implies that $(\s,\sigma)$ and $(\t,\tau)$ are well defined morphisms in $\DMan$.
A careful calculation shows that the remaining maps are also morphisms of Dirac structures.
It remains to show that the each gauge equation in the above table holds.

The equations from (G1-G3) follow immediately by definition.
The equation for (G4) holds since
\[ \u^*(\iota) = \u^*(\tau - \sigma) = 0. \]
The first equality follows from (G3) while the second follows from the fact that $\tau-\sigma$ is multiplicative.

Next we show (G5) by computing directly.
\begin{align*}
{((\u \circ \t) \times \Id)}^* \mu  &={((\u \circ \t) \times \Id)}^*( \pr_1^* \sigma + \pr_2^* \sigma - \m^* \sigma) \\
                                  &= {(\u \circ \t)}^* \sigma + \sigma - {(\m ((\u \circ \t) \times \Id))}^* \sigma \\
                                  &= {(\u \circ \t)}^* \sigma + \sigma - \sigma = {(\u \circ \t)}^* \sigma \, .
\end{align*}
It follows from the multiplicativity of $\tau - \sigma$ that
\begin{equation}
\m^* \tau + \mu = \pr_1^* \tau + \pr_2^* \tau \, .
\end{equation}
By using this expression for $\mu$ we can show (G6) by a calculation essentially identical to (G5).
Next up, we show (G7):
\begin{align*}
{(\i \times \Id)}^* \mu - \i^* \sigma &= {(\i \times \Id)}^* (\pr_1^* \sigma + \pr_2^* \sigma - \m^* \sigma) - \i^* \sigma \\
&= \i^* \sigma + \sigma - {(\u \circ \s)}^* \sigma - \i^* \sigma \\
&= -\s^* \u^* \sigma + \sigma = \s^* \iota + \sigma
\end{align*}
Since (G8) is similar we can proceed to (G9).
The gauge equation for (G9) is
\begin{equation}\label{eqn:assocgauge}
\begin{aligned}
{(\pr_1 \times \pr_2)}^* \mu + {(\m \circ (\pr_1 \times \pr_2) \times \pr_3)}^* \mu &= \\
{(\pr_2 \times \pr_3)}^* \mu + {(\pr_1 \times \m \circ (\pr_2 \times \pr_3))}^* \mu & \,.
\end{aligned}
\end{equation}
If we apply the substitution $\mu = \pr_1^* \sigma + \pr_2^* \sigma - \m^* \sigma$ throughout, we get:
\begin{align*}
\pr_1^* \sigma + \pr_2^* \sigma-{(\pr_1 \times \pr_2)}^*\m^* \sigma + {(\pr_1 \times \pr_2)}^*\m^* \sigma + \pr_3^* \sigma - \mathbf{A}_L^* \sigma &= \\
\pr_2^* \sigma + \pr_3^* \sigma-{(\pr_2 \times \pr_3)}^*\m^* \sigma + {(\pr_2 \times \pr_3)}^*\m^* \sigma + \pr_1^* \sigma - \mathbf{A}_R^* \sigma \, . &
\end{align*}
Here $\mathbf{A}_L, \mathbf{A}_R: \G^{(3)} \to \G$ are the left and right hand associativity maps.
Since $\G$ is a Lie groupoid and assumed to be associative, it follows immediately that (9) holds.
\qed\end{proof}
\bibliography{articledb}

\begin{bibdiv}
\begin{biblist}

\bib{PXstacks}{article}{
      author={Behrend, Kai},
      author={Xu, Ping},
       title={Differentiable stacks and gerbes},
        date={2011},
        ISSN={1527-5256},
     journal={J. Symplectic Geom.},
      volume={9},
      number={3},
       pages={285\ndash 341},
         url={http://projecteuclid.org/euclid.jsg/1310388899},
      review={\MR{2817778}},
}

\bib{BDiracintro}{incollection}{
      author={Bursztyn, Henrique},
       title={A brief introduction to {D}irac manifolds},
        date={2013},
   booktitle={Geometric and topological methods for quantum field theory},
   publisher={Cambridge Univ. Press, Cambridge},
       pages={4\ndash 38},
      review={\MR{3098084}},
}

\bib{BImf}{article}{
      author={Bursztyn, Henrique},
      author={Cabrera, Alejandro},
       title={Multiplicative forms at the infinitesimal level},
        date={2012},
        ISSN={0025-5831},
     journal={Math. Ann.},
      volume={353},
      number={3},
       pages={663\ndash 705},
         url={http://dx.doi.org/10.1007/s00208-011-0697-5},
      review={\MR{2923945}},
}

\bib{integrationoftwisteddiracbrackets}{article}{
      author={Bursztyn, Henrique},
      author={Crainic, Marius},
      author={Weinstein, Alan},
      author={Zhu, Chenchang},
       title={Integration of twisted {D}irac brackets},
        date={2004},
        ISSN={0012-7094},
     journal={Duke Math. J.},
      volume={123},
      number={3},
       pages={549\ndash 607},
         url={http://dx.doi.org/10.1215/S0012-7094-04-12335-8},
      review={\MR{2068969}},
}

\bib{BStacky}{article}{
      author={Bursztyn, Henrique},
      author={Noseda, Francesco},
      author={Zhu, Chengchang},
       title={Principal actions of stacky {L}ie groupoids},
        date={2015},
      eprint={https://arxiv.org/abs/1510.09208},
        note={preprint},
}

\bib{BPic}{article}{
      author={Bursztyn, Henrique},
      author={Weinstein, Alan},
       title={Picard groups in {P}oisson geometry},
        date={2004},
        ISSN={1609-3321},
     journal={Mosc. Math. J.},
      volume={4},
      number={1},
       pages={39\ndash 66, 310},
      review={\MR{2074983}},
}

\bib{Cint}{article}{
      author={Crainic, Marius},
      author={Fernandes, Rui~Loja},
       title={Integrability of {L}ie brackets},
        date={2003},
        ISSN={0003-486X},
     journal={Ann. of Math. (2)},
      volume={157},
      number={2},
       pages={575\ndash 620},
         url={http://dx.doi.org/10.4007/annals.2003.157.575},
      review={\MR{1973056}},
}

\bib{lecturesonint}{incollection}{
      author={Crainic, Marius},
      author={Fernandes, Rui~Loja},
       title={Lectures on integrability of {L}ie brackets},
        date={2011},
   booktitle={Lectures on {P}oisson geometry},
      series={Geom. Topol. Monogr.},
      volume={17},
   publisher={Geom. Topol. Publ., Coventry},
       pages={1\ndash 107},
      review={\MR{2795150}},
}

\bib{PMCT2}{article}{
      author={Crainic, Marius},
      author={Fernandes, Rui~Loja},
      author={Martinez-Torres, David},
       title={Regular {P}oisson {M}anifolds of {C}ompact {T}ypes ({PMCT} 2)},
      eprint={http://arxiv.org/abs/1603.00064v1},
         url={http://arxiv.org/abs/1603.00064v1},
}

\bib{PMCT1}{article}{
      author={Crainic, Marius},
      author={Fernandes, Rui~Loja},
      author={Torres, David~Martinez},
       title={Poisson {M}anifolds of {C}ompact {T}ypes ({PMCT} 1)},
      eprint={http://arxiv.org/abs/1510.07108v2},
         url={http://arxiv.org/abs/1510.07108v2},
}

\bib{DiracLieGroups}{article}{
      author={Li-Bland, David},
      author={Meinrenken, Eckhard},
       title={Dirac {L}ie groups},
        date={2014},
        ISSN={1093-6106},
     journal={Asian J. Math.},
      volume={18},
      number={5},
       pages={779\ndash 815},
         url={http://dx.doi.org/10.4310/AJM.2014.v18.n5.a2},
      review={\MR{3287003}},
}

\bib{Mstacks}{article}{
      author={Metlzer, David},
       title={Topological and smooth stacks},
        date={2003},
      eprint={https://arxiv.org/abs/math/0306176v1},
         url={https://arxiv.org/abs/math/0306176v1},
}

\bib{MMbook}{book}{
      author={Moerdijk, I.},
      author={Mrcun, J.},
       title={Introduction to foliations and lie groupoids},
      series={Cambridge Studies in Advanced Mathematics},
   publisher={Cambridge University Press},
        date={2003},
        ISBN={9781139438988},
}

\bib{Ortiz}{article}{
      author={Ortiz, Cristi\'an},
       title={Multiplicative {D}irac structures},
        date={2013},
        ISSN={0030-8730},
     journal={Pacific J. Math.},
      volume={266},
      number={2},
       pages={329\ndash 365},
         url={http://dx.doi.org/10.2140/pjm.2013.266.329},
      review={\MR{3130627}},
}

\bib{Morping}{article}{
      author={Xu, Ping},
       title={Morita equivalence of {P}oisson manifolds},
        date={1991},
        ISSN={0010-3616},
     journal={Comm. Math. Phys.},
      volume={142},
      number={3},
       pages={493\ndash 509},
         url={http://projecteuclid.org/euclid.cmp/1104248717},
      review={\MR{1138048}},
}

\end{biblist}
\end{bibdiv}
\end{document}